\documentclass[12pt]{article}
\usepackage{amsmath,amssymb,amsthm,threeparttable}
\usepackage{graphicx}
\usepackage{overpic}

\normalsize\textwidth140mm\hoffset=-5mm%
\normalsize\parindent=1em%
\theoremstyle{plain} \newtheorem{theorem}{Theorem}
\theoremstyle{plain} \newtheorem{lemma}{Lemma}[section]
\theoremstyle{remark} 
\theoremstyle{plain} 
\theoremstyle{plain} 
\theoremstyle{plain} \newtheorem{claim}[lemma]{Claim}
\theoremstyle{definition} \newtheorem{note}{Note}[section]
\theoremstyle{definition} \newtheorem{step}{Step}
\newcommand{\rmthree}{\mathrm{I\hspace{-1pt}I\hspace{-1pt}I}}
\newcommand{\rmfour}{\mathrm{I\hspace{-1pt}V}}
\newcommand{\rmfive}{\mathrm{V}}
\newcommand{\rmsix}{\mathrm{V\hspace{-1pt}I}}
\newcommand{\saxes}{\mathrm{Span}_{\mathbb{F}}\{a_i\mid i\in\mathbb{Z}\}}
\newcommand{\adim}{D}
\newcommand{\vdim}{d}
\newcommand{\ch}{\mathrm{ch}}
\newcommand{\eisp}[2]{M({#1},{#2})}
\newcommand{\threealg}[3]{\rmthree({#1},{#2},{#3})}
\newcommand{\fourone}[2]{\rmfour_1({#1},{#2})}
\newcommand{\fourtwo}[3]{\rmfour_2({#1},{#2},{#3})}
\newcommand{\fourthree}{\rmfour_3(\frac{1}{2},2)}
\newcommand{\fiveone}[1]{\rmfive_1({#1},\frac{5{#1}-1}{8})}
\newcommand{\fivetwo}[1]{\rmfive_2({#1},\frac{1}{2})}
\newcommand{\sixone}[1]{\rmsix_1({#1},\frac{#1}{2})}
\newcommand{\sixtwo}[1]{\rmsix_2({#1},\frac{-{#1}^2}{4(2{#1}-1)})}
\newcommand{\infalg}{\mathrm{Z}(2,\frac{1}{2})}
\newcommand{\ifusion}{\mathfrak{V}(4,3)}
\newcommand{\chf}[1]{\ch\mathbb{#1}}

\makeatletter
 \@addtoreset{equation}{section}
 
 \makeatother

\title{On the classification of 2-generated axial algebras of Majorana type}
\author{Takahiro Yabe}

\begin{document}
\maketitle

\begin{abstract}
\noindent
A class of axial algebras generated by two axes with eigenvalues 0, 1, $\eta$ and $\xi$ called axial algebras of Majorana type is introduced and classified when they are 2-generated, over fields of characteristics neither 2 nor 5 and there exists an automorphism switching generating axes..
The class includes deformations of nine Norton-Sakuma algebras.
Over fields of characteristics 5, the axial algebras of Majorana type with the axial dimension at most 5 are clasified.
\end{abstract}

\section{Introduction}
In \cite{s07}, S.\ Sakuma considered pairs of Ising vectors of Griess algebras of vertex operater algebras and proved that any subalgebra generated by such a pair is isomorphic to one of the nine Norton-Sakuma algebras \cite{c85}.
In \cite{i09}, A.\ A.\ Ivanov axiomatzed Majorana algebras, which are a generalization of the 196,884 dimensional Griess algebra called the monstrous Griess algebra, and in \cite{ipss10}, A.\ A.\ Ivanov et al.\ proved that a 2-generated Majorana algebra is isomorphic to one of those algebras.
The purpose of this paper is to study the clasification of the case when the eigenvalues $\frac{1}{4}, \frac{1}{32}$ in Majorana algebras are replaced by arbitrary scalars $\xi,\eta$ in an arbitrary base field.

A language suitable for such generalization is given by the concept of an axial algebra introduced by J.\ I.\ Hall et al.\ in \cite{hrs13}.
An axial algebra is a commutative nonassociative algebra generated by \textit{axes}, a distinguished set of idempotents, subject to a condition called a \textit{fusion rule} given in terms of the eigenvalues of the axes. 
For a fusion rule $\mathcal{F}$ and an algebra $A$, an idempotent $a$ of $A$ is called an {\it $\mathcal{F}$-axis}\/ if its action on $A$ is semisimple and the eigenspaces obey the fusion rule $\mathcal{F}$.
An algebra generated by $\mathcal{F}$-axes is called an {\it $\mathcal{F}$-axial algebra}.
A Majorana algebra is a $\ifusion$-axial algebra with an associative and positive definite bilinear form, where $\ifusion$ is a fusion rule with the set of eigenvalues $\{0,1,\frac{1}{4},\frac{1}{32}\}$ and is the fusion rule coming from the fusion rules of the Ising model in conformal field theory.
See \cite{hrs13} for details.

Let us consider an arbitrary pair $(\xi,\eta)$ of elements of the base field.
For such a pair, let $\mathcal{F}(\xi,\eta)$ be the fusion rule givrn in table \ref{tabfrule}.
The fusion rule $\ifusion$ is nothing else than $\mathcal{F}(\frac{1}{4},\frac{1}{32})$.

In \cite{mat03}, A.\ Matsuo considered the case when the $\xi$-eigenspaces of axes are $0$ for Griess algebras.  
In \cite{hrs15}, J.\ I.\ Hall, F.\ Rehren and S.\ Shpectorov generalized the result for a class of axial algebras called the primitive axial algebra of Jordan type, the eigenvalues of whose axis are $\{0,1,\eta\}$, and classified them when they are generated by 2 axes.
In \cite{r14}, F.\ Rehren deformed Norton-Sakuma algebras to the axial algebras such that the eigenvalues of their axes are $\{0,1,\xi,\eta\}$.

In the present paper, we consider a class of axial algebras over a field, which we call 2-generated \textit{axial algebras of Majorana type} $(\xi,\eta)$ and classify them in the case when there exists an automorphism switching 2 generating axes. We call such an automorphism a \textit{flip} following \cite{hrs13}.
Here, $\xi$ and $\eta$ are arbitrary elements of the base field.

The 2-generated axial algebras considered in \cite{r14} and \cite{hrs15} are particular cases of axial algebras of Majorana type $(\xi,\eta)$, but there exist other axial algebras of Majorana type.
So the classification done in this paper is a generalization of the classification of \cite{hrs13}.

\begin{table}[h]
\begin{center}
\begin{tabular}{c|cccc}
$\star$&0&1&$\xi$&$\eta$\\ \hline
0&0&0&$\xi$&$\eta$\\ 
1&0&0&$\xi$&$\eta$\\ 
$\xi$&$\xi$&$\xi$&$\{0,1\}$&$\eta$\\ 
$\eta$&$\eta$&$\eta$&$\eta$&$\{0,1,\xi\}$
\end{tabular}
\caption{The fusion rule $\mathcal{F}(\xi,\eta)$}
\label{tabfrule}
\end{center}
\end{table}

In Section 2, we introduce axial algebras of Majorana type $(\xi,\eta)$ and study their fundamental properties after recalling relevant terminologies on general axial algebras. 
In Section 3, we give a list of axial algebras of Majorana type $(\xi,\eta)$ admitting a flip.
In Section 4, we state our main theorem that the algebras listed in Section 3 exaust all the axial algebras of Majorana type $(\xi,\eta)$ admitting a flip if the base field is of characteristic neither 2 nor 5 and that such a result holds for characteristic 5 when the algebra satiffies a condition about linear relations among axes.
We also give an outline of their proofs.
Finally, we give the details of the proofs in Section 5.

The results of this paper were presented in RIMS Workshop `Research on algebraic combinatorics, related groups and algebras,' held in Kyoto in December, 2019. 
The author noticed that the algebra $\infalg$  given in this paper was independently constructed in a recent preprint \cite{fms20}. 

\section{Axial algebras of Majorana type}
Let $\mathbb{F}$ be a field of characteristic not 2.

\subsection{Axial algebra of Majorana type}
For distinct elements $\xi$ and $\eta$ of $\mathbb{F}\setminus\{0,1\}$, an \textit{axial algebras of Majorana type} $(\xi,\eta)$ is a commutative nonassociative algebra algebra $M$ equipped with a set of generators $A$ such that for all $a\in A$, the following conditions are satisfied.
\begin{itemize}
\item[\rm{(1)}]$M=\bigoplus_{\alpha\in\{1,0,\xi,\eta\}}\eisp{a}{\alpha}$, where $\eisp{a}{\alpha}=\{w\in M\mid aw=\alpha w\}$ for all $a\in M$ and $\alpha\in\mathbb{F}$.
\item[\rm{(2)}]$\eisp{a}{1}=\mathbb{F}a$.
\item[\rm{(3)}]$\eisp{a}{\alpha}\eisp{a}{\beta}\subset\bigoplus_{\gamma\in\alpha\star\beta}\eisp{a}{\gamma}$ for all $\alpha,\beta\in\{0,1,\xi,\eta\}$, where $\star$ is the fusion rule $\mathcal{F}(\xi,\eta)$ of table \ref{tabfrule}.
\end{itemize}
The fusion rule $\mathcal{F}(\xi,\eta)$ is given in \cite{r14}.

When a non-trivial permutation of axes of 2-generated $\mathcal{F}(\xi,\eta)$-axial algebra $M$ induces an automorphism of $M$, we call such an automorphism a \textit{flip} (See \cite{hrs13}).
The flip is unique if it exists.
We denote it by $\theta$.

Let $a_0$ and $a_1$ are generating axes of $M$.
For each $i\in\{0,1\}$, let $\tau_i$ denote the automorphism called the \textit{Miyamoto involution}\/ which is $1$ on $\eisp{a_i}{1}\oplus \eisp{a_i}{0}\oplus \eisp{a_i}{\xi}$ and $-1$ on $\eisp{a_i}{\eta}$ (See \cite{mi96}).

\subsection{Axial dimension and linear relation of axes}
Let $M$ be an axial algebra of Majorana type $(\xi,\eta)$ generated by $\{a_0,a_1\}$ admitting a flip $\theta$. 
Recall the Miyamoto involutions $\tau_0$ and $\tau_1$. 

For each integer $i\in\mathbb{Z}$, we set $a_{2i}=(\tau_1\circ\tau_0)^i(a_0)$ and $a_{2i+1}=(\tau_1\circ\tau_0)^i(a_1)$.
Then $a_i$ is an axis, $\theta(a_i)=a_{1-i}$, $\tau_0(a_i)=a_{-i}$, and $\tau_1(a_i)=a_{2-i}$ for all $i\in\mathbb{Z}$.
The group $G=\langle\tau_0,\theta\rangle$ is a dihedral group.
Let $f_i\in G$ be an automorphism such that $f_i=(\theta\circ\tau_0)^i$.

We call the dimension of the subspace $\saxes$ the \textit{axial dimension} of $M$ and denote it by $\adim=\mathrm{Adim}\,M$, while the dimension of the whole algebra $M$ by $\vdim=\dim M$.

Let $\{i_k\}_{k\in\mathbb{Z}_{>0}}$ be the unique sequence of integers such that $i_1=0$, $i_{2k}=1-i_{2k-1}$ and $i_{2k+1}=-i_{2k}$ for all $k\in\mathbb{Z}_{>0}$.
Then $i_{2k}=k$, $i_{2k-1}=1-k$, $a_{i_{2k}}=\theta(a_{i_{2k-1}})$ and $a_{i_{2k+1}}=\tau_0(a_{i_{2k}})$.

Linear combinations of axes are described by the following two lemmas.

\begin{lemma}\label{lem1}
\sl
$\{a_i\mid1\leq i\leq\adim\}$ is a basis of $\saxes$.
\end{lemma}
\begin{proof}
Set $k\in\mathbb{Z}$.
Assume that $a_{i_{k+1}}=\sum_{j=1}^k\alpha_ja_{i_j}$ holds for some $\alpha_1,\ldots,\alpha_k\in\mathbb{F}$.
Once $a_{i_{k+2}}\in\mathit{Span}_{\mathbb{F}}\{a_{i_j}\mid1\leq j\leq k\}$ is shown, then it follows that $a_i\in\mathrm{Span}_{\mathbb{F}}\{a_{i_j}\mid1\leq j\leq k\}$ for all $i$ by induction, and hence the proof is completed. 

If $k=2l$ for some $l\in\mathbb{Z}$, then $a_{i_{k+2}}=a_{1-i_{k+1}}=\theta(a_{i_{k+1}})=\sum_{j=1}^k\alpha_j\theta(a_{i_j})=\sum_{j=1}^k\alpha_ja_{1-i_j}$.
Since $\{i_1,\ldots,i_{k}\}=\{-l+1,\ldots,l\}$, $1-i_j\in\{i_1,\ldots,i_k\}$ for all $1\leq j\leq k$ and  $a_{i_{k+2}}\in\mathit{Span}_{\mathbb{F}}\{a_{i_j}\mid1\leq j\leq k\}$.

If $k=2l+1$ for some $l\in\mathbb{Z}$, then $a_{i_{k+2}}=a_{-i_{k+1}}=\tau_0(a_{i_{k+1}})=\sum_{j=1}^k\alpha_j\tau_0(a_{i_j})=\sum_{j=1}^k\alpha_ja_{-i_j}$.
Since $\{i_1,\ldots,i_k\}=\{-l,\ldots,l\}$ $-i_j\in\{i_1,\ldots,i_k\}$ for all $1\leq j\leq k$ and $a_{i_{k+2}}\in\mathit{Span}_{\mathbb{F}}\{a_{i_j}\mid1\leq j\leq k\}$.
\end{proof}

\begin{lemma}
\sl
If $\adim<\infty$, then there exist $k\in\mathbb{Z}_{>0}$ and $\alpha_1,\ldots,\alpha_k\in\mathbb{F}$ such that one of the following properties holds:
\begin{itemize}
\item[\rm{(1)}]$\adim=2k-1$, $\alpha_k\neq0$ and $\sum_{j=1}^{k}\alpha_j(a_{i_{2j-1}}+a_{i_{2j}})=0$.
\item[\rm{(2)}]$\adim=2k-1$, $\alpha_k\neq0$ and $\sum_{j=1}^{k}\alpha_j(a_{i_{2j-1}}-a_{i_{2j}})=0$.
\item[\rm{(3)}]$\adim=2k$, $\alpha_k\neq0$ and $\sum_{j=2}^{k}\alpha_j(a_{i_{2j-1}}+a_{i_{2j-2}})+\alpha_1 a_0=0$.
\item[\rm{(4)}]$\adim=2k$, $\alpha_k\neq0$  and  $\sum_{j=1}^{k} \alpha_j(a_{i_{2j+1}}-a_{i_{2j}})=0$.
\end{itemize}
\end{lemma}

\begin{proof}
By Lemma \ref{lem1}, there exist $\beta_1,\ldots,\beta_{\adim+1}\in\mathbb{F}$ such that $\beta_{\adim+1}\neq0$ and $\sum_{j=1}^{\adim+1}\beta_j a_{i_j}=0$.
If $\adim=2k-1$ for some $k\in\mathbb{Z}_{>0}$, then 
\begin{eqnarray*}
0=\theta(0)+0=\sum_{j=0}^{k}(\beta_{2j-1}+\beta_{2j})(a_{i_2j}+a_{i_2j-1}).
\end{eqnarray*}
If $\beta_{2k-1}+\beta_{2k}\neq0$, then the property (1) holds.
Otherwise, all of the coefficients are zero since the dimension of $\saxes$ becomes less than $\adim$ if there exists a non-zero coefficient.
Hence the property (2) holds.

In the case when $\adim=2k$, by considering $\tau_0(0)+0$, it follows that 
$$\sum_{j=1}^{k}(\beta_{2j+1}+\beta_{2j})(a_{i_2j}+a_{i,2j+1})+2\beta_1a_0=0.$$
If $\beta_{2k+1}+\beta_{2k}\neq0$, then the property (3) holds.
Otherwise, all of the coefficients are zero and the property (4) holds by a similar argument as above.
\end{proof}

We say that $M$ satisfies an \textit{even relation} when the property (1) or (3) holds and $M$ satisfies an \textit{odd relation} otherwise.

\subsection{Some useful elements of $M$}
 We follow the notations introduced in the preceding subsection.

Consider the element 
$$p_{i,j}=a_ja_{i+j}-\eta(a_i+a_{i+j})$$
 for $i,j\in\mathbb{Z}$.
Then $p_{1,j}=a_ja_{j+1}-\eta(a_j+a_{j+1})$ is invariant under $G$ and $p_{2,0}=a_ja_{j+2}-\eta(a_j+a_{j+2})$ and $p_{2,1}=a_ja_{j+2}-\eta(a_j+a_{j+2})$ are invariant under $\tau_0$ and switched by the flip $\theta$. 
In other words, 
$$
\def\arraystretch{1.5}
\begin{array}{ll}
\hbox{$\tau_i(p_{1,j})=p_{1,j}$, $\tau_i(p_{2,0})=p_{2,0}$, $\tau_i(p_{2,1})=p_{2,1}$,}\cr
\hbox{$\theta(p_{1,j})=p_{1,j}$, $\theta(p_{2,0})=p_{2,1}$, and $\theta(p_{2,1})=p_{2,0}$}
\end{array}
$$
 for $i=0,1$ (See \cite{r14}, section 3).
The element $p_1=p_{1,0}=a_0a_1-\eta(a_0+a_1)$ will play a particular role later.

For each integer $i\in\mathbb{Z}$, let $\varphi_i:M\to\mathbb{F}$ be the linear map such that $\varphi_i(x)a_i$ is the projection of $x\in M$ to $\eisp{a_i}{1}=\mathbb{F}a_i$.
We set $\lambda_i=\varphi_0(a_i)$.

Set 
$$
\def\arraystretch{1.5}
\begin{array}{l}
\hbox{$x_i=p_{i,0}-(\lambda_i-\eta)a_0+\frac{\eta}{2}(a_{-i}+a_i)$,}\cr
\hbox{$y_i=a_i-a_{-i}$,}\cr
\hbox{$z_i=p_{i,0}-((1-\xi)\lambda_i-\eta)a_0-\frac{\xi-\eta}{2}(a_{-i}+a_i)$}.\cr
\end{array}
$$
Then $z_ia_0=0$, $x_ia_0=\xi x_i$ and $y_ia_0=\eta y_i$. (See \cite{r14}, section 3).

Let us consider the multiplication of $p_{i,j}$.
Since $f_j(z_ia_0)=0$,
\begin{eqnarray}
\label{multiaipij}
a_jp_{i,j}=(\xi-\eta)p_{i,j}+&&((1-\xi)\lambda_i+\eta(\xi-\eta-1))a_j \nonumber \\
&&+\frac{1}{2}\eta(\xi-\eta)(a_{i+j}+a_{j-i})
\end{eqnarray}
for all $i,j\in\mathbb{Z}_{\geq0}$.

\begin{lemma}\label{lem2}
\sl
Let $q\in M$ and $\pi\in\mathbb{F}$.
If $a_i\in\eisp{q}{\pi}$ for all $i\in\mathbb{Z}$, then $p_{i,j}\in\eisp{q}{\pi}$ for all $i,j\in\mathbb{Z}$.
\end{lemma}
In order to prove this lemma, we use the property known as the \textit{Seress condition} (cf.\ \cite{hrs13}, Section 3): 
for $w\in M$ and $z\in \eisp{a_i}{0}$, $a_i(wz)=(a_iw)z$.
This property is indeed satisfied for a $(\xi,\eta)$-axial algebra of Majorana type $M$ since $\eisp{a_i}{0}\eisp{a_i}{k}\subset\eisp{a_i}{k}$ holds for all $k\in\{0,1,\xi,\eta\}$ by the definition of the fusion rule $\mathcal{F}(\xi,\eta)$. 
\begin{proof}[Proof of Lemma \ref{lem2}]
If $a_i\in\eisp{q}{\pi}$ for all $i\in\mathbb{Z}$, then $q-\pi a_j\in \eisp{a_j}{0}$ and hence 
\begin{eqnarray*}
0
&=&(a_ja_{i+j})(q-\pi a_j)-a_j(a_{i+j}(q-\pi a_j))\\
&=&p_{i,j}q+\pi\eta(a_j+a_{i+j})-\pi(a_ia_{i+j})a_j-\pi a_j(a_{i+j}-a_ja_{i+j})=p_{i+j}q-\pi p_{i,j}.
\end{eqnarray*}
for all $i,j\in\mathbb{Z}$.
\end{proof}

Recall the element $p_1=p_{1,0}=a_0a_1-\eta (a_0+a_1)$.

\begin{lemma}\label{lem3}\sl
If $\dim M\geq3$, then $p_1\neq0$.
\end{lemma}
\begin{proof}
If $p_1=0$, then $a_0a_1\in\mathbb{F}a_0+\mathbb{F}a_1$, which yields $M=\langle a_0,a_1\rangle_{alg}\subset\mathbb{F}a_0+\mathbb{F}a_1$ and hence $\dim M\leq2$.
\end{proof}

Recal the elemants $p_{2,0}=a_0a_2-\eta(a_0+a_2)$ and $p_{2,1}=a_{-1}a_{1}-\eta(a_{-1}+a_{1})$.

\begin{lemma}
If $\xi=2\eta$, then 
\begin{eqnarray}
\label{multia0p21`}
p_{2,0}\in\sum_{i=-1}^\mathbb{F}a_i+\mathbb{F}p_1-\frac{\eta}{2}(a_2+a_{-2}).
\end{eqnarray}

If $\xi\neq2\eta$, then
\begin{eqnarray}
\label{multia0p21}
a_0p_{2,1}\in&&\mathbb{F}a_0 \nonumber \\
&&+\frac{1}{\xi-2\eta}((2(1-2\xi)\lambda_1+2\xi^2-\xi\eta+2\eta^2-\frac{\xi}{2}-\eta)(2p_1+\eta(a_1+a_{-1}))\nonumber \\
&&+\eta(\xi-\eta)(2p_{2,0}+\eta(a_2+a_{-2})))
\end{eqnarray}
and 
\begin{eqnarray}
&p_{2,1}&-p_{2,0} \nonumber \\
&\in&\mathbb{F}(a_1-a_0)+\frac{\xi-4\eta}{4}(a_3-a_{-2}) \nonumber \\
&&-\frac{(2\xi^2-12\xi\eta-2\xi+8\eta)\lambda_1+5\xi\eta^2+\xi^2\eta+\xi\eta-6\eta^2}{4\eta(\xi-\eta)}(a_2-a_{-1}).
\end{eqnarray}

In particular,
\begin{eqnarray}
\label{p2linrel}
p_{2,1}-p_{2,0}\in\mathbb{F}(a_1-a_0)+\mathbb{F}(a_2-a_{-1})+\frac{\xi-4\eta}{4}(a_3-a_{-2})
\end{eqnarray}
in the both cases.
\end{lemma}

\begin{proof}
By the fusion rule, $x_1x_1-z_1z_1-\varphi_0(x_1x_1)a_0)\in\eisp{a_0}{0}$.

When $\xi=2\eta$ the coefficient of $p_{2,1}$ of $x_1x_1-z_1z_1-\varphi_0(x_1x_1)a_0)$ is 0.
Thus we can compute $a_0(x_1x_1-z_1z_1-\varphi_0(x_1x_1)a_0)$ and then  
\begin{eqnarray*}
p_{2,0}\in\sum_{i=-1}^\mathbb{F}a_i+\mathbb{F}p_1-\frac{\eta}{2}(a_2+a_{-2}).
\end{eqnarray*}
holds since the multiplication is 0.

When $\xi\neq2\eta$, then
\begin{eqnarray}
\label{multia0p21}
a_0p_{2,1}\in&&\mathbb{F}a_0 \nonumber \\
&&+\frac{1}{\xi-2\eta}((2(1-2\xi)\lambda_1+2\xi^2-\xi\eta+2\eta^2-\frac{\xi}{2}-\eta)(2p_1+\eta(a_1+a_{-1}))\nonumber \\
&&+\eta(\xi-\eta)(2p_{2,0}+\eta(a_2+a_{-2})))
\end{eqnarray}
since $a_0(x_1x_1-z_1z_1-\varphi_0(x_1x_1)a_0)=0$.

By the fusion rule, $p_{i,0}p_{j,0}=\frac{1}{\xi}a_0(z_ix_j-z_iz_j)-(z_ix_j-p_{i,0}p_{j,0})$.
By the flip-invariance of $p_1p_1$,
\begin{eqnarray}
&p_{2,1}&-p_{2,0} \nonumber \\
&\in&\mathbb{F}(a_1-a_0)+\frac{\xi-4\eta}{4}(a_3-a_{-2}) \nonumber \\
&&-\frac{(2\xi^2-12\xi\eta-2\xi+8\eta)\lambda_1+5\xi\eta^2+\xi^2\eta+\xi\eta-6\eta^2}{4\eta(\xi-\eta)}(a_2-a_{-1})
\end{eqnarray}
if $\xi\neq2\eta$.
\end{proof}

By the simular argument as the proof above, we can decide $a_0(p_{i.j}+p_{i,-j})$ for any $i,j\in\mathbb{Z}$.

\section{List of axial algebra of Majorana type admitting flip}

\subsection{The primitive axial algebras of Jordan type}
An axial algebra $M$ of Majorana type $(\xi,\eta)$ admitting a flip with $\eisp{a_0}{\xi}=0$ or $\eisp{a_0}{\eta}=0$ is a 2-generated primitive axial algebras of Jordan type $\eta$ or $\xi$, respectively (\cite{hrs15}).
The following lemma is given in \cite{hrs15}, Proposition 4.6, Proposition 4.7 and Proposition 4.8.
\begin{lemma}[J.I.\ Hall et al.]
\sl
\label{lemhall}
If $M$ is a primitive axial algebra of Jordan type $\eta$ generated by two axes $a_0,a_1$, then the following three properties hold:
\begin{itemize}
\item[\rm{(1)}]If $M$ is not isomorphic to the associative algebra $\mathbb{F}$ or $\mathbb{F}\oplus\mathbb{F}$, then $\eta=\frac{1}{2}$ or $\lambda_1=\frac{\eta}{2}$.
\item[\rm{(2)}]If $\dim M=2$ and $M$ is not isomorphic to $\mathbb{F}\oplus\mathbb{F}$, then $\eta=-1$ or $\frac{1}{2}$ and $a_0a_1=\eta(a_0+a_1)$.
\item[\rm{(3)}] $p_1w=((1-\eta)\lambda_1-\eta)w$ for all $w\in M$.
\end{itemize}
\end{lemma}

\subsection{Algebra $\threealg{\xi}{\eta}{\alpha}$ and its quotients}

For $\xi,\eta,\alpha\in\mathbb{F}$ with $2\xi\neq1$, let $\threealg{\xi}{\eta}{\alpha}$ be the $\mathbb{F}$-space $\mathbb{F}\hat{q}\oplus\bigoplus_{i=-1}^1\mathbb{F}\hat{a}_i$ with the multiplication given by
\begin{itemize}
\item$\hat{a}_i^2=\hat{a}_i$ for all $i$.
\item$\hat{a}_i\hat{a}_{i+1}=\hat{q}+\frac{(\xi-\eta)}{2}((\alpha+1)\hat{a}_0+\hat{a}_1+\hat{a}_{-1})+\eta(\hat{a}_i+\hat{a}_{i+1})$ for $i=0,-1$.
\item$\hat{a}_{-1}\hat{a}_1=(1-\alpha)(\hat{q}+\frac{(\xi-\eta)}{2}((\alpha+1)\hat{a}_0+\hat{a}_1+\hat{a}_{-1}))+\eta(\hat{a}_1+\hat{a}_{-1})$. 
\item$\hat{q}w=\frac{-\xi(\xi+1)(\xi-\eta)\alpha-\xi(3\xi^2+3\xi\eta-\eta-1)}{4(2\xi-1)}w$ for all $w\in\threealg{\xi}{\eta}{\alpha}$.
\end{itemize}
The parameters such that $\threealg{\xi}{\eta}{\alpha}$ becomes an axial algebra of Majorana type $(\xi,\eta)$ and admits a flip are as follows:
\begin{itemize}
\item[(i)]$\alpha=0$ and $0$, $1$, $\xi$ and $\eta$ are distinct.
\item[(ii)]$\eta=\frac{1}{2}$ and $\xi\notin\{0,1,\frac{1}{2}\}$.
\end{itemize}

The quotients of these algebras not of Jordan type are given by:
\begin{itemize}
\item[(i)]$\threealg{\xi}{\frac{1-3\xi^2}{3\xi-1}}{0}/\mathbb{F}\hat{q}$, where $3\xi\neq\pm1$ and $3\xi^2\neq1$.
\item[(ii)]$\threealg{-1}{\frac{1}{2}}{\alpha}/\mathbb{F}\hat{q}$, where $\ch\mathbb{F}\neq3$.
\item[(iii)]$\threealg{\xi}{\frac{1}{2}}{-3}/\mathbb{F}\hat{q}$.
\end{itemize}
We denote these quotients by $\threealg{\xi}{\frac{1-3\xi^2}{3\xi-1}}{0}^{\times}$,  $\threealg{-1}{\frac{1}{2}}{\alpha}^{\times}$ and  $\threealg{\xi}{\frac{1}{2}}{-3}^{\times}$, respectively.

\subsection{Algebra $\fourone{\xi}{\eta}$ and its quotients}
For $\xi,\eta\in\mathbb{F}$, let $\fourone{\xi}{\eta}$ be the $\mathbb{F}$-space $\mathbb{F}\hat{q}\oplus\bigoplus_{i=-1}^2\mathbb{F}\hat{a}_i$ with the multiplication given by
\begin{itemize}
\item$\hat{a}_i\hat{a}_i=\hat{a}_i$ for all $i$.
\item$\hat{a}_i\hat{a}_{j}=\hat{q}+\frac{\xi-\eta}{2}(\hat{a}_0+\hat{a}_1+\hat{a}_{-1}+\hat{a}_2)+\eta(\hat{a}_i+\hat{a}_{j})$ if $|i-j|=1$ or $3$.
\item$\hat{a}_i\hat{a}_{i+2}=0$ for $i=-1,0$
\item$\hat{q}w=\frac{-2\xi\eta-\xi+\eta}{2}w$ for all $w\in\fourone{\xi}{\eta}$.
\end{itemize}

The parameters such that $\fourone{\xi}{\eta}$ is an axial algebra of Majorana type $(\xi,\eta)$ and admits a flip are listed as follows:
\begin{itemize}
\item[(i)]$\xi=\frac{1}{4}$, $\chf{F}\neq3$ and $\eta\notin\{0,1,\frac{1}{4}\}$
\item[(ii)]$\eta=\frac{\xi}{2}$ and $\xi\notin\{0,1,2\}$.
\end{itemize}

The quotients of these algebras with axial dimension 4 are as follows:
\begin{itemize}
\item[(i)] $\fourone{\frac{1}{4}}{\frac{1}{2}}/\mathbb{F}\hat{q}$ where $\chf{F}\neq3$.
\item[(ii)]$\fourone{-\frac{1}{2}}{-\frac{1}{4}}/\mathbb{F}\hat{q}$ where $\chf{F}\neq3$ nor $5$.
\end{itemize}
We denote these quotients by $\fourone{\frac{1}{4}}{\frac{1}{2}}^{\times}$ and $\fourone{-\frac{1}{2}}{-\frac{1}{4}}^{\times}$, respectively.

\subsection{Algebra $\fourtwo{\xi}{\eta}{\mu}$ and its quotient}
For $\xi,\eta,\mu\in\mathbb{F}$ such that $\mu\neq0$, let $\fourtwo{\xi}{\eta}{\mu}$ be the $\mathbb{F}$-space $\mathbb{F}\hat{p}_1\oplus\bigoplus_{i=-1}^2\mathbb{F}\hat{a}_i$ with multiplication given by
\begin{itemize} 
\item$\hat{p}_1\hat{p}_1=(2\eta^2-\eta\xi-\frac{1}{2}\eta+\frac{\xi-2\xi^2}{2\mu})\hat{p}_1+(\frac{\eta^3\mu}{2}+\eta^3-\frac{\xi\eta^2}{2}+\frac{\xi\eta-2\xi^2\eta}{4\mu})(\hat{a}_0+\hat{a}_1+\hat{a}_{-1}+\hat{a}_2)$.
\item$\hat{p}_1\hat{a}_i=\frac{\xi-\eta}{2}(2\hat{p}_1+\eta(\hat{a}_{i-1}+\hat{a}_{i+1}))+(-\mu\eta^2+\frac{\eta(2\xi-2\eta-1)}{2})\hat{a}_0$ for all $i$, where $\hat{a}_{-2}=\hat{a}_2$ and $\hat{a}_3=\hat{a}_{-1}$.
\item$\hat{a}_i\hat{a}_i=\hat{a}_i$ for all $i=-1,0,1,2$.
\item$\hat{a}_i\hat{a}_{j}=\hat{p}_1+\eta(\hat{a}_i+\hat{a}_{j})$ if $|i-j|=1$ or $3$.
\item$\hat{a}_i\hat{a}_{i+2}=r_{2,i}+\eta(\hat{a}_i+\hat{a}_{i+2})$ for $i=-1,0$, where $r_{2,i}=-\mu(2\hat{p}_1+\eta(\hat{a}_{i+1}+\hat{a}_{i-1}))-\eta(\hat{a}_i+\hat{a}_{i+2})$ and $\hat{a}_{3}=\hat{a}_{-1}$.
\end{itemize}
The parameters such that $\fourtwo{\xi}{\eta}{\mu}$ is an axial algebra of Majorana type $(\xi,\eta)$ and admits a flip are as follows:
\begin{itemize}
\item[(i)]$(\xi,\mu)=(\frac{1}{2},\frac{1-4\eta}{2\eta})$ and $\eta\notin\{0,1,\frac{1}{2}\}$.
\item[(ii)] $(\eta,\mu)=(\frac{\xi^2}{2},\frac{1}{\xi})$ and $\xi\notin\{0,1,2,\pm\sqrt{2}\}$,
\item[(iii)] $(\eta,\mu)=(\frac{1-\xi^2}{2},\frac{-1}{\xi+1})$ and  $\xi\notin\{0,\pm1,\pm\sqrt{-1},-1\pm\sqrt{2}\}$.
\end{itemize}

The only quotient of these algebras with axial dimension 4 is $\fourtwo{-1}{\frac{1}{2}}{-1}/\mathbb{F}(\hat{p}_1+\frac{3}{8}(\hat{a}_0+\hat{a}_1+\hat{a}_{-1}+\hat{a}_2))$ where $\chf{F}\neq3$.
We denote this quotient by $\fourtwo{-1}{\frac{1}{2}}{-1}^{\times}$.

\subsection{Algebra $\fourthree$ and its quotient}
Let $\fourthree$ be the $\mathbb{F}$-space $\mathbb{F}\hat{q}\oplus\bigoplus_{i=-1}^2\mathbb{F}\hat{a}_i$ with multiplication given by
\begin{itemize}
\item For all $w\in\fourthree$, $\hat{q}w=0$,
\item For al $i$, $\hat{a}_i^2=\hat{a}_i$
\item $\hat{a}_i\hat{a}_{i+1}=\hat{q}-\frac{1}{2}(2\hat{a}_0+2\hat{a}_1+\hat{a}_{-1}+\hat{a}_2)+2(\hat{a}_i+\hat{a}_{i+1})$ for all $i$ where $\hat{a}_i=\hat{a}_{i+4}+\hat{a}_{i+3}-\hat{a}_{i+1}$.
\item $\hat{a}_i\hat{a}_{i+2}=\hat{a}_i-\hat{a}_{i+1}+\hat{a}_{i-1}$ for all $i$ where $\hat{a}_i=\hat{a}_{i+4}+\hat{a}_{i+3}-\hat{a}_{i+1}$.
\end{itemize}
This is an axial algebra of Majorana type $(\frac{1}{2},2)$ generated by $\{\hat{a}_0,\hat{a}_1\}$ with $\lambda_1=\lambda_2=1$ and admits a flip. 
The only quotient of $\fourthree$ which is not given yet is $\fourthree/\mathbb{F}\hat{q}$.
We denote it by $\fourthree^{\times}$.

\subsection{Algebras $\fiveone{\xi}$, $\fivetwo{\xi}$ and their quotients}
For $\xi\in\mathbb{F}\setminus\{0,1,\frac{-1}{3},\frac{1}{5},\frac{9}{5}\}$, let $\fiveone{\xi}$ be the $\mathbb{F}$-space $\mathbb{F}\hat{p}_1\oplus\bigoplus_{i=-2}^2\mathbb{F}\hat{a}_i$ with multiplication given by
\begin{itemize}
\item$\hat{p}_1\hat{p}_1=\frac{-5(3\xi+1)(5\xi-1)}{128}\hat{p}_1-\frac{(7\xi-3)(3\xi+1)(5\xi-1)}{2048}(\hat{a}_2+\hat{a}_{-2}+\hat{a}_1+\hat{a}_{-1}+\hat{a}_0)$.
\item$\hat{p}_1\hat{a}_i=\frac{3\xi+1}{8}\hat{p}_1+\frac{(5\xi-1)(9\xi-25)}{256}\hat{a}_i+\frac{(3\xi+1)(5\xi-1)}{128}(\hat{a}_{i+1}+\hat{a}_{i-1})$ for all $i$, where $\hat{a}_3=\hat{a}_{-2}$ and $\hat{a}_{-3}=\hat{a}_2$.
\item$\hat{a}_i\hat{a}_i=\hat{a}_i$ for all $i$.
\item$\hat{a}_i\hat{a}_j=\hat{p}_1+\frac{5\xi-1}{8}(\hat{a}_i+\hat{a}_j)$ if $|i-j|=1$ or $4$.
\item$\hat{a}_i\hat{a}_j=-\hat{p}_1-\frac{5\xi-1}{16}(\hat{a}_2+\hat{a}_{-2}+\hat{a}_1+\hat{a}_{-1}+\hat{a}_0)+\frac{5\xi-1}{8}(\hat{a}_i+\hat{a}_j)$ if $|i-j|=2$ or $3$.
\end{itemize}
This is an axial algebra of Majorana type $(\xi,\eta)$ and admits a flip.

For $\xi\in\mathbb{F}\setminus\{0,1,\frac{1}{2}\}$, let $\fivetwo{\xi}$ be the $\mathbb{F}$-space $\mathbb{F}\hat{p}_1\oplus\bigoplus_{i=-2}^2\mathbb{F}\hat{a}_i$ with multiplication given by
\begin{itemize}
\item$\hat{p}_1\hat{p}_1=\frac{(2\xi-1)(2\xi-3)}{32}(\hat{a}_2+\hat{a}_{-2}-4(\hat{a}_1+\hat{a}_{-1})+6\hat{a}_0),$
\item$\hat{p}_1\hat{a}_i=\frac{2\xi-1}{8}(4\hat{p}_1-2\hat{a}_i+\hat{a}_{i+1}+\hat{a}_{i-1})$ for all $i$, where $\hat{a}_{i+5}=\hat{a}_{i}+5(\hat{a}_{i+4}-\hat{a}_{i+1})-10(\hat{a}_{i+3}-\hat{a}_{i+2})$.
\item$\hat{a}_i\hat{a}_i=\hat{a}_i$ for all $i$
\item$\hat{a}_i\hat{a}_{i+1}=\hat{p}_1+\eta(\hat{a}_i+\hat{a}_{i+1})$ for all $i$ where $\hat{a}_{i+5}=\hat{a}_{i}+5(\hat{a}_{i+4}-\hat{a}_{i+1})-10(\hat{a}_{i+3}-\hat{a}_{i+2})$.
\item$\hat{a}_i\hat{a}_{i+2}=4\hat{p}_1-\frac{1}{4}(\hat{a}_2+\hat{a}_{-2}-4(\hat{a}_1+\hat{a}_{-1})+6\hat{a}_0)+\eta(\hat{a}_i+\hat{a}_{i+2})$ for all $i$, where $\hat{a}_{i+5}=\hat{a}_{i}+5(\hat{a}_{i+4}-\hat{a}_{i+1})-10(\hat{a}_{i+3}-\hat{a}_{i+2})$.
\end{itemize}
This is a axial algebras of Majorana type $(\xi,\eta)$ and admits a flip. 

The quotient of these algebras not given yet is $\fivetwo{\xi}/\mathbb{F}(\hat{a}_2+\hat{a}_{-2}-4(\hat{a}_1+\hat{a}_{-1})+6\hat{a}_0)$.
We denote this quotient by $\fivetwo{\xi}^{\times}$.

\subsection{Algebras $\sixone{\xi}$, $\sixtwo{\xi}$ and their quotients}
For $\xi\in\mathbb{F}\setminus\{0,1,2\}$, let $\sixone{\xi}$ be the $\mathbb{F}$-space $\mathbb{F}\hat{q}\oplus\mathbb{F}\hat{p}_1\oplus\bigoplus_{i=-2}^3\mathbb{F}\hat{a}_i$ with multiplication given by
\begin{itemize} 
\item$\hat{p}_1\hat{p}_1=\frac{\xi^2}{16}(\hat{q}-2\hat{p}_1+\frac{\xi}{2}(\hat{a}_3+\hat{a}_2+\hat{a}_{-2}+\hat{a}_1+\hat{a}_{-1}+\hat{a}_0))-\frac{\xi^2+2\xi}{8}\hat{p}_1$.
\item$\hat{q}w=-\frac{7\xi^2+2\xi}{4}w$ for all $w\in\sixone{\xi}$.
\item$\hat{p}_1\hat{a}_i=\frac{\xi}{2}\hat{p}_1-\frac{\xi^2}{4}\hat{a}_i+\frac{\xi^2}{8}(\hat{a}_{i+1}+\hat{a}_{i-1})$
\item$\hat{a}_i\hat{a}_i=\hat{a}_i$ for all $i$
\item$\hat{a}_i\hat{a}_j=\hat{p}_1+\frac{\xi}{2}(\hat{a}_i+\hat{a}_j)$ if $|i-j|=1$ or $5$.
\item$\hat{a}_i\hat{a}_j=-\frac{\xi}{4}(\hat{a}_2+\hat{a}_{-2}+\hat{a}_0)+\frac{\xi}{2}(\hat{a}_i+\hat{a}_j)$ if $|i-j|=2$ or $4$ and $i$ is even.
\item$\hat{a}_i\hat{a}_j=-\frac{\xi}{4}(\hat{a}_3+\hat{a}_{-1}+\hat{a}_1)+\frac{\xi}{2}(\hat{a}_i+\hat{a}_j)$ if $|i-j|=2$ or $4$ and $i$ is odd.
\item$\hat{a}_i\hat{a}_j=\hat{q}-2\hat{p}_1+\frac{\xi}{2}(\hat{a}_3+\hat{a}_2+\hat{a}_{-2}+\hat{a}_1+\hat{a}_{-1}+\hat{a}_0)
+\frac{\xi}{2}(\hat{a}_i+\hat{a}_j)$ if $|i-j|=3$.
\end{itemize}
This is an axial algebras of Majorana type $(\xi,\frac{\xi}{2})$ and admits a flip.

For $\xi\in\mathbb{F}\setminus\{0,1,\frac{4}{9},\frac{2}{5},-4\pm2\sqrt{5}\}$, let $\sixtwo{\xi}$ be the $\mathbb{F}$-space $\mathbb{F}\hat{q}\oplus\mathbb{F}\hat{p}_1\oplus\bigoplus_{i=-2}^3\mathbb{F}\hat{a}_i$ with multiplication given by
\begin{itemize}
\item Set
\begin{eqnarray*}
r_{2,i}&=&\hat{q}-\frac{(3\xi-2)(5\xi-2)^2}{\xi^2(9\xi-4)}\hat{p}_1+\frac{(3\xi-2)(5\xi-2)}{8(2\xi-1)}(\hat{a}_{i+3}+\hat{a}_{i+1}+\hat{a}_{i-1})\\
&&+\frac{21\xi^2-18\xi+4}{8(2\xi-1)}(r_{i+2}+r_{i-2}+\hat{p}_i).
\end{eqnarray*}
Then $\hat{a}_i\hat{a}_j=r_{2,i}+\eta(\hat{a}_i+\hat{a}_j)$ if $|i-j|=2$ or $4$ and
\begin{eqnarray*}
\hat{p}_1\hat{p}_1&=&\frac{\xi^2(39\xi^2-22\xi+2)}{16(2\xi-1)^2}\hat{p}_1-\frac{\xi^4(9\xi-4)}{32(2\xi-1)^2(5\xi-2)}r_{2,0}\\
&&-\frac{\xi^4(3\xi-1)(9\xi-4)}{128(2\xi-1)^3}(\hat{a}_3+\hat{a}_1+\hat{a}_{-1})\\
&&-\frac{\xi^4(3\xi-1)(9\xi-4)}{64(2\xi-1)^2(5\xi-2)}(\hat{a}_2+\hat{a}_{-2}+\hat{a}_0).
\end{eqnarray*}
\item
$\hat{q}w=-\frac{(3\xi-2)(5\xi-2)(12\xi^2-\xi-2)}{8(2\xi-1)(9\xi-4)}w$ for all $w\in\sixtwo{\xi}$,
\item$\hat{p}_1\hat{a}_i=(\xi-\eta)\hat{p}_1+((1-\xi)\lambda_1+\eta(\xi-\eta-1))\hat{a}_i+\frac{\eta(\xi-\eta)}{2}(\hat{a}_{i+1}+\hat{a}_{i-1})$ for all $i$ where $\hat{a}_4=\hat{a}_{-2}$, $\hat{a}_{-3}=\hat{a}_3$ and $\lambda_1=\frac{-\xi^2(3\xi-2)}{16(2\xi-1)^2}$.
\item$\hat{a}_i\hat{a}_i=\hat{a}_i$ for all $i$.
\item$\hat{a}_i\hat{a}_j=\hat{p}_1+\eta(\hat{a}_i+\hat{a}_j)$ if $|i-j|=1$ or $5$.
\item$\hat{a}_i\hat{a}_j=r_{3,i}+\eta(\hat{a}_i+\hat{a}_j)$ if $|i-j|=3$, where
\begin{eqnarray*}
r_{3,i}&=&\frac{2(2\xi-1)}{\xi}(2\hat{p}_1-\frac{\xi^2}{4(2\xi-1)}(\hat{a}_{i+1}+\hat{a}_{i-1}))\\
&&-\frac{2(2\xi-1)}{5\xi-2}(2r_{2,i}-\frac{\xi^2}{4(2\xi-1)}(\hat{a}_{i+2}+\hat{a}_{i-2}))\\
&&+\frac{\xi^2}{4(2\xi-1)}\hat{a}_{i+3}+\frac{\xi(29\xi^2-22\xi+4)}{4(2\xi-1)(5\xi-2)}\hat{a}_i,
\end{eqnarray*}
and $\hat{a}_{i+6}=\hat{a}_i$.
\end{itemize}
This algebra is an axial algebras of Majorana type $(\xi,\eta)$ and admits a flip.

The quotients of these algebras not given yet are as follows:
\begin{itemize}
\item[(i)]$\rmsix_1(\frac{-2}{7},\frac{-1}{7})/\mathbb{F}\hat{q}$ where $\chf{F}\neq7$.
\item[(ii)]$\rmsix_2(\frac{2}{3},\frac{-1}{3})/\mathbb{F}\hat{q}$ where $\chf{F}\neq3$.
\item[(iii)]$\rmsix_2(\frac{1\pm\sqrt{97}}{24},\frac{53\pm5\sqrt{97}}{192})/\mathbb{F}\hat{q}$ where $\mathbb{F}\ni\sqrt{97}$ and $\chf{F}\neq3$ nor $11$
\item[(iv)]$\rmsix_2(2,7)/\mathbb{F}\hat{q}$ where $\chf{F}=11$.
\end{itemize}
We denote these algebras by $\rmsix_1(\frac{-2}{7},\frac{-1}{7})^{\times}$, $\rmsix_2(\frac{2}{3},\frac{-1}{3})^{\times}$, $\rmsix_2(\frac{1\pm\sqrt{97}}{24},\frac{53\pm5\sqrt{97}}{192})^{\times}$ and $\rmsix_2(2,7)^{\times}$, respectively.

\subsection{Algebra $\infalg$}
Let $\infalg$ be the $\mathbb{F}$-space $\bigoplus_{i\in\mathbb{Z}_{>0}}\mathbb{F}\hat{p}_i\oplus\bigoplus_{i\in\mathbb{Z}}\mathbb{F}\hat{a}_i$ with multiplication given by:
\begin{itemize}
\item$\hat{p}_i\hat{p}_j=\frac{3}{4}(\hat{p}_i+\hat{p}_j)-\frac{3}{8}(\hat{p}_{i+j}+\hat{p}_{|i-j|})$ for all $i$ and $j$.
\item$\hat{a}_i\hat{a}_j=\hat{p}_{|i-j|}+\frac{1}{2}(\hat{a}_i+\hat{a}_j)$ for all $i$ and $j$, where $p_0=0$.
\item$\hat{a}_i\hat{p}_j=\frac{3}{2}\hat{p}_j-\frac{3}{4}\hat{a}_i+\frac{3}{8}(\hat{a}_{i-j}+\hat{a}_{i+j})$.
\end{itemize}

\section{Main result}
Our main result is the following theorem.
\begin{theorem}
\sl
An axial algebra of Majorana type $(\xi,\eta)$ admitting a flip over a field of characteristic neither $2$ nor $5$ is isomorphic to a primitive axial algebra of Jordan type, a quotient of $\infalg$, or one of the algebras listed in Table \ref{tabalg}.
\end{theorem}

Our proof of Theorem 1 mostly works even when $\ch{F}=5$, and we obtain the following theorem.
\begin{theorem}
\sl
An axial algebra of Majorana type $(\xi,\eta)$ admitting a flip  over a field of characteristic $5$ is isomorphic to a primitive axial algebra of Jordan type, a quotient of $\infalg$, or one of the algebras listed in Table \ref{tabalg} unless $\adim\geq6$ .
\end{theorem}

\begin{threeparttable}
\begin{center}
{\small\def\arraystretch{1.5}
\tabcolsep0.5em
\begin{tabular}{|c|c|p{15em}|p{9em}|c|c|}
\hline
$\adim$&$d$&\multicolumn{1}{|c|}{Universal types $M$} & \multicolumn{1}{|c|}{Quotients $\bar{M}$ of $M$}&$\bar{D}$&$\bar{d}$  \\ 
\hline
3&4&$\threealg{\xi}{\eta}{0}$\par
\hfill with $\xi\notin\{0,1,\frac{1}{2}\}$ and $\eta\notin\{0,1,\xi\}$
& $\threealg{\xi}{\frac{1-3\xi^2}{3\xi-1}}{0}^{\times}$\par
\hfill with $\xi\notin\{\pm\frac{1}{3},\frac{1}{\sqrt{3}}\}$ 
&3&3\\ 
\cline{3-4}
&&$\threealg{\xi}{\frac{1}{2}}{\alpha}$ with $\xi\notin\{0,1,\frac{1}{2}\}$
&$\threealg{-1}{\frac{1}{2}}{\alpha}^{\times}$\par
\hfill when $\chf{F}\neq3$ \par
$\threealg{\xi}{\frac{1}{2}}{-3}^{\times}$
&&
\\ \hline
4&5&$\fourone{\frac{1}{4}}{\eta}$ \par
\hfill with $\eta\notin\{0,1,\frac{1}{4}\}$ 
when $\chf{F}\neq3$
& $\fourone{\frac{1}{4}}{\frac{1}{2}}^{\times}$
&4&4\\ 
\cline{3-4}
&&$\fourone{\xi}{\frac{\xi}{2}}$ with $\xi\notin\{0,1,2\}$
& $\fourone{-\frac{1}{2}}{-\frac{1}{4}}^{\times}$\par
\hfill when $\chf{F}\neq3,5$ 
&&
\\ 
\cline{3-6}
&&$\fourtwo{\frac{1}{2}}{\eta}{\frac{1-4\eta}{2\eta}}$ with $\eta\notin\{0,1,\frac{1}{2}\}$
&&& \\ 
\cline{3-6}
&&$\fourtwo{\xi}{\frac{\xi^2}{2}}{\frac{1}{\xi}}$ \par
\hfill with $\xi\notin\{0,1,2,\pm\sqrt{2}\}$
& $\fourtwo{-1}{\frac{1}{2}}{-1}^{\times}$\par
\hfill when $\chf{F}\neq3$
&4&4 \\ 
\cline{3-6}
&&$\fourtwo{\xi}{\frac{1-\xi^2}{2}}{\frac{-1}{\xi+1}}$\par
\hfill with $\xi\notin\{0,\pm1,\pm\sqrt{-1},-1\pm\sqrt{2}\}$
&&& \\ 
\cline{3-6}
&&$\fourthree$ when $\chf{F}\neq3$
& $\fourthree^{\times}$&4&4 \\ 
\hline
5&6&$\fiveone{\xi}$\par
\hfill with $\xi\notin\{0,1,\frac{-1}{3},\frac{1}{5},\frac{9}{5}\}$
&&& \\ 
\cline{3-6}
&&$\fivetwo{\xi}$ with $\xi\notin\{0,1,\frac{1}{2}\}$
& $\fivetwo{\xi}^{\times}$&4&5 \\ 
\hline
6&8&$\sixone{\xi}$ with $\xi\notin\{0,1,2\}$
& $\rmsix_1(\frac{-2}{7},\frac{-1}{7})^{\times}$\par
\hfill when $\chf{F}\neq7$&6&7 \\ 
\cline{3-4}
&&$\sixtwo{\xi}$\par
\hfill with $\xi\notin\{0,1,\frac{4}{9},\frac{2}{5},-4\pm2\sqrt{5}\}$
&
$\rmsix_2(\frac{2}{3},\frac{-1}{3})^{\times}$\par
\hfill when $\chf{F}\neq3$ \par
$\rmsix_2(\frac{1\pm\sqrt{97}}{24},\frac{53\pm5\sqrt{97}}{192})^{\times}$ \par
\hfill when $\chf{F}\neq3,11$\par
\hfill and $\mathbb{F}\ni\sqrt{97}$ \par
$\rmsix_2(2,-4)^{\times}$\par
\hfill when $\chf{F}=11$
&& \\
\hline
\end{tabular}
}
\end{center}
\caption{Axial algebras of Majorana type $(\xi,\eta)$ admitting a flip which are not of Jordan 
type, where $\bar{D}$ is the axial dimension of $\bar{M}$ and $\bar{d}$ is the dimension of $\bar{M}$.}
\label{tabalg}
\end{threeparttable}

Figure \ref{graphxieta} and \ref{graphxieta2} describe a graph of $(\xi,\eta)$ of $\infalg$ and 11 universal types of the table \ref{tabalg} except $\threealg{\xi}{\eta}{0}$ over $\mathbb{R}$.
The graphs of $\fourone{\frac{1}{4}}{\eta}$, $\fourtwo{\xi}{\frac{\xi^2}{2}}{\frac{1}{\xi}}$, $\fiveone{\xi}$ and $\sixtwo{\xi}$ intersect at the point $(\frac{1}{4},\frac{1}{32})$, where the  fusion rule equals $\ifusion$.

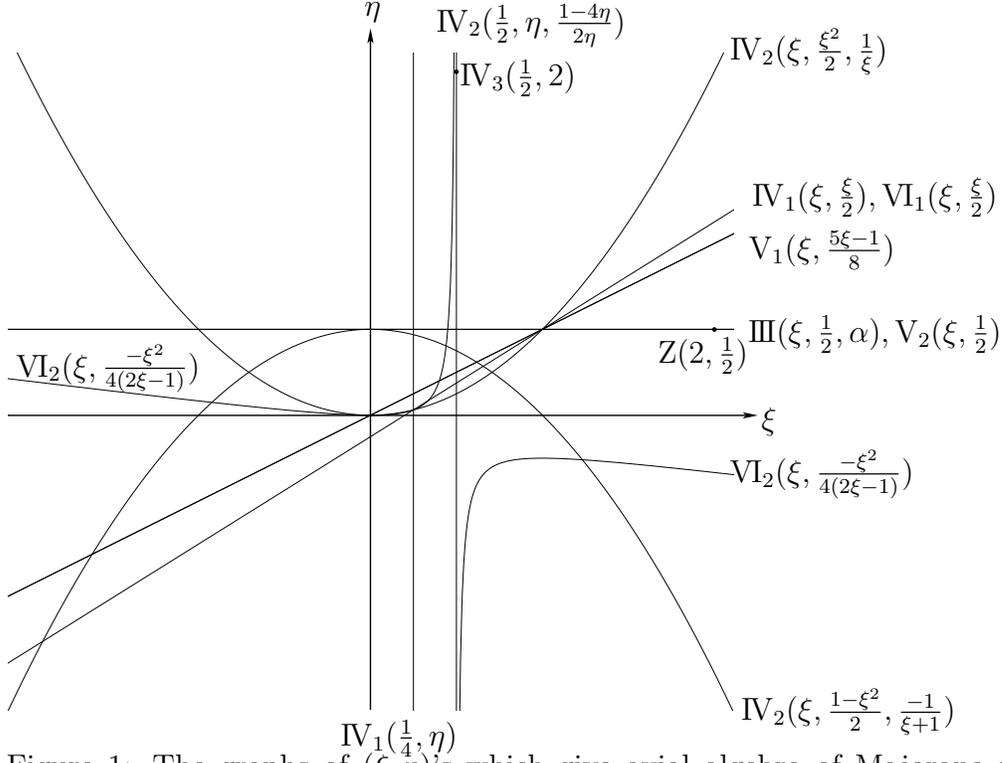
\begin{figure}
\includegraphics[width=100mm]{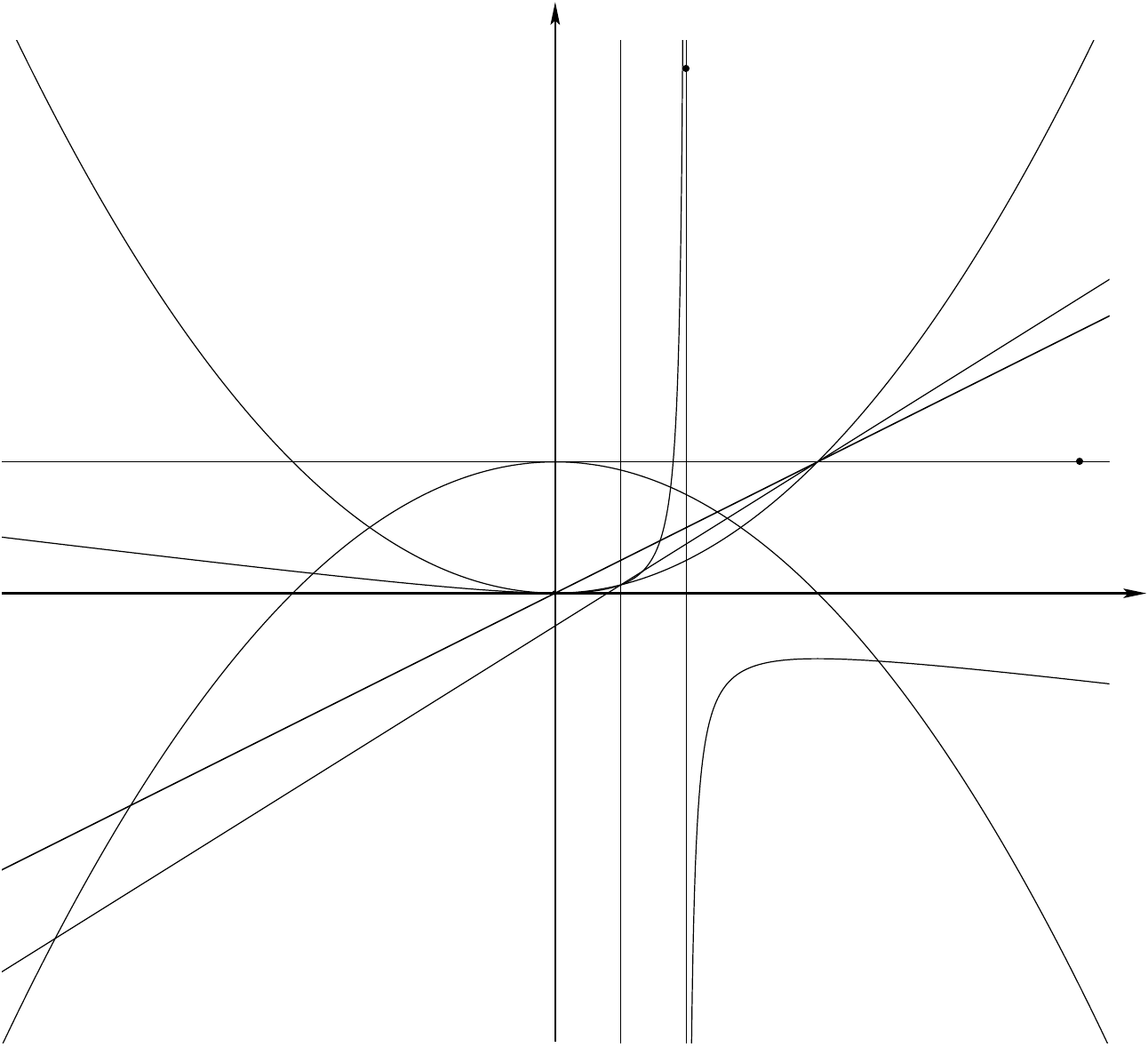}
\begin{picture}(0,0)
\put(40,143){\makebox(0,0){$\threealg{\xi}{\frac{1}{2}}{\alpha}, \fivetwo{\xi}$}}
\put(0,110){\makebox(0,0){$\xi$}}
\put(-150,265){\makebox(0,0){$\eta$}}
\put(-25,135){\makebox(0,0){$\infalg$}}
\put(40,195){\makebox(0,0){$\fourone{\xi}{\frac{\xi}{2}}, \sixone{\xi}$}}
\put(20,175){\makebox(0,0){$\fiveone{\xi}$}}
\put(15,250){\makebox(0,0){$\fourtwo{\xi}{\frac{\xi^2}{2}}{\frac{1}{\xi}}$}}
\put(-90,260){\makebox(0,0){$\fourtwo{\frac{1}{2}}{\eta}{\frac{1-4\eta}{2\eta}}$}}
\put(-95,240){\makebox(0,0){$\fourthree$}}
\put(-140,-10){\makebox(0,0){$\fourone{\frac{1}{4}}{\eta}$}}
\put(30,0){\makebox(0,0){$\fourtwo{\xi}{\frac{1-\xi^2}{2}}{\frac{-1}{\xi+1}}$}}
\put(20,90){\makebox(0,0){$\sixtwo{\xi}$}}
\put(-250,130){\makebox(0,0){$\sixtwo{\xi}$}}
\end{picture}
\caption{The graphs of $(\xi,\eta)$'s which give axial algebra of Majorana type admitting a flip.}
\label{graphxieta}
\end{figure}

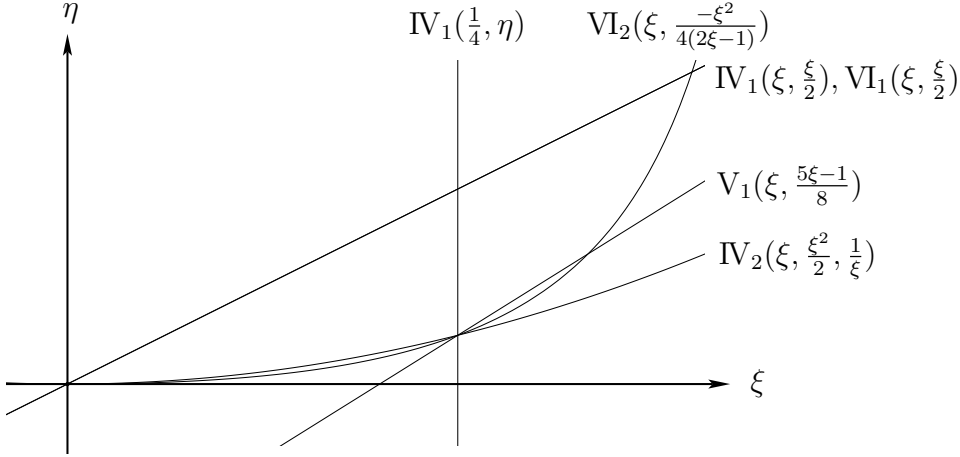
\begin{figure}
\includegraphics[width=100mm]{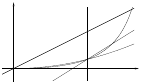}
\begin{picture}(0,0)
\put(0,30){\makebox(0,0){$\xi$}}
\put(-260,170){\makebox(0,0){$\eta$}}
\put(30,145){\makebox(0,0){$\fourone{\xi}{\frac{\xi}{2}}, \sixone{\xi}$}}
\put(12,105){\makebox(0,0){$\fiveone{\xi}$}}
\put(15,78){\makebox(0,0){$\fourtwo{\xi}{\frac{\xi^2}{2}}{\frac{1}{\xi}}$}}
\put(-110,165){\makebox(0,0){$\fourone{\frac{1}{4}}{\eta}$}}
\put(-30,165){\makebox(0,0){$\sixtwo{\xi}$}}
\end{picture}
\caption{An enlargement of Figure \ref{graphxieta} around (0,0).}
\label{graphxieta2}
\end{figure}

\begin{note}
The algebras $\threealg{\frac{1}{4}}{\frac{1}{32}}{0}$, $\fourone{\frac{1}{4}}{\frac{1}{32}}$, $\fourtwo{\frac{1}{4}}{\frac{1}{32}}{4}$, $\rmfive_1(\frac{1}{4},\frac{1}{32})$ and $\rmsix_2(\frac{1}{4},\frac{1}{32})$ are isomorphic to the Norton-Sakuma algebras $\mathrm{3A}$, $\mathrm{4A}$, $\mathrm{4B}$, $\mathrm{5A}$ and $\mathrm{6A}$, respectively (\cite{r14}).
Other Norton-Sakuma algebras are of Jordan type.
\end{note}

\begin{note}
In \cite{r14}, F.\ Rehren found the algebras $\threealg{\xi}{\eta}{0}$,  $\fourone{\frac{1}{4}}{\eta}$, $\fourtwo{\xi}{\xi^2}{\frac{1}{\xi}}$, $\fiveone{\xi}$ and $\sixtwo{\xi}$ and they call these algebras $(3A^{\prime}_{\alpha,\beta})$, $(4A_{\beta})$, $(4B_{\alpha})$, $(5A_{\alpha})$ and $(6A_{\alpha})$, respectively.
\end{note}

\begin{note}
When $M$ is isomorphic to $\threealg{\xi}{\eta}{0}$ or its quotients of axial dimension 3, $G$ is isomorphic to the dihedral group $D_6$.
When $M$ is isomorphic to $\fourone{\xi}{\eta}$, $\fourtwo{\xi}{\eta}{\mu}$ or their quotients of axial dimension 4, $G$ is isomorphic to the dihedral group $D_8$.
When $M$ is isomorphic to $\fiveone{\xi}$, $G$ is isomorphic to the dihedral group $D_{10}$.
When $M$ is isomorphic to $\fourthree$, $\fourthree^{\times}$, $\sixone{\xi}$, $\sixtwo{\xi}$, $\rmsix_1(\frac{-2}{7},\frac{-1}{7})^{\times}$, $\rmsix_2(\frac{2}{3},\frac{-1}{3})^{\times}$ or $\rmsix_2(\frac{1\pm\sqrt{97}}{24},\frac{53\pm5\sqrt{97}}{192})^{\times}$, $G$ is isomorphic to the dihedral group $D_{12}$. 
\end{note}

Let us now describe an outline of the proof.
The details will be given in the next section.

Assume that $\xi,\eta\in\mathbb{F}\setminus\{0,1\}$ are distinct, $M$ is an axial algebra of Majorana type $(\xi,\eta)$ generated by $\{a_0,a_1\}$, and $M$ admits a flip.
We prove the theorem by the following steps.
\begin{step}\label{propminf}
\sl
Suppose that $(\xi,\eta)=(2,\frac{1}{2})$ and $\lambda_1=1$.
If $\mathrm{ch}\mathbb{F}\neq5$ or $\adim\leq5$, then $M$ is a quotient of $\infalg$.
\end{step}

Assume that $M$ is not isomorphic to quotients of $\infalg$.

\begin{step}\label{prop3}
\sl
If $\adim\leq3$ and $M$ is not of Jordan type, then one of the following statements holds:
\begin{itemize}
\item[\rm{(1)}]$\xi\notin\{0,1,\frac{1}{2}\}$, $\eta\notin\{0,1,\xi\}$ and $M$ is isomorphic to $\threealg{\xi}{\eta}{0}$.
\item[\rm{(2)}]$\xi\notin\{0,1,\frac{1}{2},\pm\frac{1}{3},\frac{1}{\sqrt{3}}\}$ and $M$ is isomorphic to $\threealg{\xi}{\frac{1-3\xi^2}{3\xi-1}}{0}^{\times}$.
\item[\rm{(3)}]$\xi\notin\{0,1,\frac{1}{2}\}$ and $M$ is isomprphic to $\threealg{\xi}{\frac{1}{2}}{\alpha}$.
\item[\rm{(4)}]$\xi\notin\{0,1,\frac{1}{2}\}$ and $M$ is isomorphic to $\threealg{\xi}{\frac{1}{2}}{-3}^{\times}$.
\item[\rm{(5)}]$\chf{F}\neq3$ and $M$ is isomorphic to $\threealg{-1}{\frac{1}{2}}{\alpha}^{\times}$.
\end{itemize}
\end{step}

\begin{step}
\sl
\label{prop4e}
If $\adim=4$ and $M$ satisfies an even relation, then $\xi\notin\{0,1,\frac{1}{2}\}$ and $M$ is isomorphic to $\fivetwo{\xi}^{\times}$.
\end{step}

\begin{step}
\label{prop4o}
\sl
If $M$ is of axial dimension 4 and satisfies an odd relation, then one of the following holds:
\begin{itemize}
\item[\rm{(1)}]\label{a411}
$\chf{F}\neq3$ and $\eta\notin\{0,1,\frac{1}{4}\}$ and $M$ is isomorphic to $\fourone{\frac{1}{4}}{\eta}$.
\item[\rm{(2)}]\label{a411q}
$\chf{F}\neq3$ and $M$ is isomorphic to $\fourone{\frac{1}{4}}{\frac{1}{2}}^{\times}$.
\item[\rm{(3)}]\label{a412}
$\xi\notin\{0,1,2\}$ and $M$ is isomorphic to $\fourone{\xi}{\frac{\xi}{2}}$.
\item[\rm{(4)}]\label{a412q}
$\chf{F}\neq3,5$ and $M$ is isomorphic to $\fourone{-\frac{1}{2}}{-\frac{1}{4}}^{\times}$. 
\item[\rm{(5)}]\label{a421}
$\eta\notin\{0,1,\frac{1}{2}\}$ and $M$ is isomorphic to $\fourtwo{\frac{1}{2}}{\eta}{\frac{1-4\eta}{2\eta}}$.
\item[\rm{(6)}]\label{a422}
$\xi\notin\{0,1,2,\pm\sqrt{2}\}$ and $M$ is isomorphic to $\fourtwo{\xi}{\frac{\xi^2}{2}}{\frac{1}{\xi}}$.
\item[\rm{(7)}]\label{a423}
$\xi\notin\{0,\pm1,\pm\sqrt{-1},-1\pm\sqrt{2}\}$ and $M$ is isomorphic to $\fourtwo{\xi}{\frac{1-\xi^2}{2}}{\frac{-1}{\xi+1}}$.
\item[\rm{(8)}]\label{a422q}
$\chf{F}\neq3$ and $M$ is isomorphic to $\fourtwo{-1}{\frac{1}{2}}{-1}^{\times}$.
\item[\rm{(9)}]\label{a43}
$\chf{F}\neq3$ and $M$ is isomorphic to $\fourthree$ or its quotient $\fourthree^{\times}$.
\end{itemize}
\end{step}

\begin{step}
\label{prop5}
\sl
If $\adim=5$, then one of the following holds:.
\begin{itemize}
\item[\rm{(1)}]$\xi\notin\{0,1,\frac{-1}{3},\frac{1}{5},\frac{9}{5}\}$ and $M$ is isomorphic to $\fiveone{\xi}$.
\item[\rm{(2)}]$\xi\notin\{0,1,\frac{1}{2}\}$ and $M$ is isomorphic to $\fivetwo{\xi}$.
\end{itemize}
\end{step}

\begin{step}
\label{prop6}
\sl
If $\adim\geq6$ and $\chf{F}\neq5$, then one of the following holds:
\begin{itemize}
\item[\rm{(1)}] $\xi\notin\{0,1,2\}$ and $M$ is isomorphic to $\sixone{\xi}$.
\item[\rm{(2)}]$\chf{F}\neq7$ and $M$ is isomorphic to $\rmsix_1(\frac{-2}{7},\frac{-1}{7})^{\times}$.
\item[\rm{(3)}]$\xi\notin\{0,1,\frac{4}{9},\frac{2}{5},-4\pm2\sqrt{5}\}$ and $M$ is isomorphic to $\sixtwo{\xi}$.
\item[\rm{(4)}]$\chf{F}\neq3$ and $M$ is isomorphic to $\rmsix_2(\frac{2}{3},\frac{-1}{3})^{\times}$.
\item[\rm{(5)}]$\mathbb{F}\ni\sqrt{97}$, $\chf{F}\neq3,11$ and $M$ is isomorphic to $\rmsix_2(\frac{1\pm\sqrt{97}}{24},\frac{53\pm5\sqrt{97}}{192})^{\times}$.
\item[\rm{(6)}]$\chf{F}=11$ and $M$ is isomorphic to $\rmsix_2(2,-4)^{\times}$.
\end{itemize}
\end{step}

\section{Proof of the theorem}
Let $M$ be an axial algebra of Majorana type $(\xi,\eta)$ generated by $\{a_0,a_1\}$ over a field $\mathbb{F}$ of characteristic not 2 admitting a flip.
Recall the notations in Section 2 and recall that $f_i\in G$ denotes the automorphism $f_i=(\theta\circ\tau_0)^i$.

\subsection{Proof of Step \ref{propminf}}
Assume that $(\xi,\eta,\lambda_1)=(2,\frac{1}{2},1)$ and either $\chf{F}\neq5$ or $\adim\le5$ holds.

\begin{claim}
\sl
$p_{i,j}=p_{i,0}$ and $\lambda_i=1$ for all $i\in\mathbb{Z}_{>0}$ and $j\in\mathbb{Z}$.
\end{claim}
\begin{proof}
First, we consider the case when  $\mathrm{ch}\mathbb{F}\neq5$.
We prove the claim by induction on $i$.
If $i=1$, then $p_{1,0}=p_{i,j}$ for all $j$.
Assume that $p_{j,l}=p_{j,0}$ for all $l$ and $\lambda_j=1$ if $j<i$.
Since $a_0(x_kx_{i-k}-z_kz_{i-k}-\varphi_0(x_kx_{i-k})a_0)=0$, for $0<k<i$, there exist $\rho_k\in\mathbb{F}$ such that  $a_0(p_{i,k}+p_{i,-k})=3p_{i,0}+2\rho_ka_0+\frac{3}{4}(a_i+a_{-i})$. 
This suffices if $k=0$.
\begin{eqnarray*}
p_1p_{i-1,0}&&=\frac{1}{2}a_0(z_1x_{i-1}-z_1z_{i-1})-(z_1x_{i-1}-p_1p_{i-1,0})\\
&&=(-\frac{3}{4}\rho_1+\frac{9}{16})a_0+\frac{3}{4}(p_1+p_{i-1,0})-\frac{3}{8}p_{i-2,0}+\frac{3}{16}(p_{i,1}+p_{i,-1}-4p_{i,0}).
\end{eqnarray*}
Since $f_i(p_{i-1,0})=p_{i-1,0}$ and $f_i(p_{i-2,0})=p_{i-2,0}$ by the inductive hypothesis, $\rho_1=\frac{3}{4}$ or $a_0=a_i$.
If $a_0=a_i$, $p_{i,j}=0$ for all $j$.
So we may assume $a_0\neq a_i$ and $\rho_1=\frac{3}{4}$.
Then, since $f_k(p_1p_{i-1,0})=p_1p_{i-1,0}$, $p_{i,1}+p_{i.-1}-4p_{i,0}=p_{i,k+1}+p_{i,k-1}-4p_{i,k}$.
If $\mathbb{F}[\sqrt{3}]\ni(2+\sqrt{3})^i\neq1$, then $p_{i,k}=p_{i,0}$ for all $k$.

If $(2+\sqrt{3})^i=1$, then
\begin{eqnarray*}
p_{i,k}=&&(2-\frac{(2+\sqrt{3})^k+(2-\sqrt{3})^k}{2})p_{i,0}\\
&&+\frac{(1+\sqrt{3})(2+\sqrt{3})^k-(1-\sqrt{3})(2-\sqrt{3})^k-2\sqrt{3}}{4\sqrt{3}}p_{i,1}\\
&&+\frac{(1+\sqrt{3})(2-\sqrt{3})^k-(1-\sqrt{3})(2+\sqrt{3})^k-2\sqrt{3}}{4\sqrt{3}}p_{i,-1}\\
\end{eqnarray*}
and $i\geq5$.
Since 
\begin{eqnarray*}
p_{2,0}p_{i-2,0}&&=\frac{1}{2}a_0(z_2x_{i-2}-z_2z_{i-2})-(z_2x_{i-2}-p_2p_{i-2,0})\\
&&=-(\frac{3}{4}\rho_2+\frac{9}{16})a_0+\frac{3}{4}(p_2+p_{i-2,0})-\frac{3}{8}p_{i-4,0}+\frac{3}{16}(p_{i,2}+p_{i,-2}-4p_{i,0}),
\end{eqnarray*}
$0=p_{i,2}+p_{i.-2}-4p_{i,0}-(p_{i,3}+p_{i,-1}-4p_{i,1})=10(p_{i,0}-p_{i,1})$ and then $p_{i,k}=p_{i,0}$ for all $k$.

If $p_{i,1}=p_{i,0}$, then $\rho_1=-\lambda_i+\frac{1}{4}$.
Since $\rho_1=\frac{3}{4}$,  $\lambda_i=1$.

In the case when $\adim\leq 5$ and $\ch\mathbb{F}=5$, we prove it by induction about $i$, too.
$p_{2,1}=p_{2,0}$ and $\lambda_2=1$ by the same calculation as the case when $\ch\mathbb{F}\neq5$.
Since there exist $\alpha_{-2}\ldots\alpha_{2}\in\mathbb{F}$ such that $a_3=\sum_{j=-2}^2\alpha_j a_j$ in this case,  $a_i=f_{i-3}(a_3)=\sum_{j=-2}^2\alpha_j a_{i+j-3}$.
Set $i\geq3$ and assume that $p_{k,l}=p_{k,0}$ for all $l$ if $k<i$.
Then $p_{i,0}=\sum_{j=-2}^{2}\alpha_j p_{|i+j-3|,0}+(\sum_{j=-2}^2\alpha_j-1)\eta a_0$.
Since $f_i(p_{i,0})=p_{i,0}$, $\sum_{i=-2}^2\alpha_i=1$ or $a_0=a_i$ by the inductive hypothesis.
So $p_{i,k}=p_{i,0}$. 
By the same calculation as the case when $\ch\mathbb{F}\neq5$, $\lambda_i=1$ holds.
\end{proof}

By this claim, $p_{i,0}p_{j,0}=\frac{1}{2}a_0(z_ix_j-z_iz_j)-(z_ix_j-p_ip_j)=\frac{3}{4}(p_{i,0}+p_{j,0})-\frac{3}{8}(p_{i+j,0}+p_{|i-j|,0})$ for all $i,j\in\mathbb{Z}$.
Thus there exists a surjective homomorphism $\infalg$ to $M$ such that $\hat{a}_i\mapsto a_i$ and $\hat{p}_i\mapsto p_{i,0}$.
Thus $M$ is a quotient of $\infalg$.

\subsection{Proof of Step \ref{prop3}}
From this subsection, we assume that $M$ is not isomorphic to quotients of $\infalg$.
By Step \ref{propminf}, $(\xi,\eta,\lambda_1)\neq(2,\frac{1}{2},1)$ if $\chf{F}\neq5$ or $\adim\le5$.

Assume $\adim\le3$.
\begin{claim}
\sl
\label{lem32}
$\adim=3$ and $\dim M\geq3$.
\end{claim}
\begin{proof}
Since $M$ is not of Jordan type, $\eisp{a_0}{\eta},\eisp{a_0}{\xi}\neq0$.
Thus $\dim M\geq3$.

If $\adim=1$, then $M\cong\mathbb{F}$ and it is of Jordan type.
Hence it suffices to verify that when $\adim=2$, $M$ is of Jordan type. 

When $M$ satisfies an odd relation, $\tau_0(w)=w$ for all $w\in M=\langle a_0,a_1\rangle_{alg}$ and $\eisp{a_0}{\eta}=0$ since $a_1=a_{-1}=\tau_0(a_1)$.
Hence $M$ is of Jordan type in this case.
Therefore it suffices to consider the case when $M$ satisfies an even relation.
Then, $a_1+a_{-1}+\alpha a_0=0$ for some $\alpha\in\mathbb{F}$ and 
$$0=a_0(a_1+a_{-1}+\alpha a_0)=2p_1+(\alpha(1-\eta)+2\eta)a_0.$$
By the $G$-invariance of $p_1$, $\alpha(1-\eta)+2\eta=0$ and then $p_1=0$.
By Lemma \ref{lem3}, $\dim M<3$ and then $M$ is of Jordan type. 
\end{proof}

Hence it suffices to consider the case with $\adim=3$.
\begin{claim}
\label{lemo5}
\sl
$M$ satisfies an odd relation.
\end{claim}
\begin{proof}
Assume that $M$ satisfies an even relation.
Then there exists $\alpha\in\mathbb{F}$ such that $a_2+a_{-1}+\alpha(a_1+a_0)=0$.
Since
$$0=a_0(a_2+a_{-1}+\alpha(a_1+a_0))=p_{2,0}+(\alpha+1)p_1+(\alpha+2\eta)a_0,$$
$\alpha=-2\eta$ and $p_{2,0}=(2\eta-1)p_1$.
Then $p_1\notin\saxes$ because $p_1=0$ if $p_1\in\saxes$.
Thus
$$0=p_1(a_2+a_{-1}+\alpha(a_1+a_0))\in\saxes+2(\alpha+1)(\xi-\eta)p_1,$$
$\alpha=-1$, $\eta=\frac{1}{2}$ and $\dim M=4$.
By the $G$-invariance of ${p_1}^2$, there exists $\mu\in\mathbb{F}$ such that ${p_1}^2=\mu p_1$.
Since $\dim M=4$,
$$\mathbb{F}(p_1-(\lambda_1-\frac{1}{2})a_0+\frac{1}{4}(a_1+a_{-1}))=\eisp{a_0}{\xi},$$
$$\mathbb{F}(p_1-((1-\xi)\lambda_1-\frac{1}{2})a_0-\frac{2\xi-1}{4}(a_1+a_{-1}))=\eisp{a_0}{0}$$
and
\begin{eqnarray*}
&(&p_1-(\lambda_1-\frac{1}{2})a_0+\frac{1}{4}(a_1+a_{-1}))^2\\
&&\in\mathbb{F}a_0+ (\mu-2\xi\lambda_1+2\xi-\frac{1}{2})p_1-\frac{2\xi-1}{4}(2\lambda_1-1)(a_1+a_{-1}),
\end{eqnarray*}
$\mu=2(\xi+1)\lambda_1-\frac{4\xi+1}{2}$.
Since
\begin{eqnarray*}
&(&p_1-(\lambda_1-\frac{1}{2})a_0+\frac{1}{4}(a_1+a_{-1}))(p_1-((1-\xi)\lambda_1-\frac{1}{2})a_0-\frac{2\xi-1}{4}(a_1+a_{-1}))\\
&&=((\xi^2+\xi+2)\lambda_1-\xi^2-1)p_1\\
&&+(1-\xi)(((\xi-1)\lambda_1^2-\frac{\xi-2}{2}\lambda_1+\frac{\xi-1}{2})a_0+\frac{3-4\xi}{4}\lambda_1+\frac{\xi-1}{4})(a_1+a_{-1}))\\
&&\in \eisp{a_0}{\xi},
\end{eqnarray*}
$\lambda_1=1$ and $\xi=2$.
Thus $(\xi,\eta,\lambda_1)=(2,\frac{1}{2},1)$ in the case.
Hence $M$ satisfies an odd relation.
\end{proof}

Hence it suffices to verify the following claim.
\begin{claim}
\sl
Let $M$ satisfy an odd relation. Then $M$ is isomorphic to  $\threealg{\xi}{\eta}{0}$,  $\threealg{\xi}{\frac{1}{2}}{\alpha}$ or their quotients.
\end{claim}
\begin{proof}
There exists $\alpha\in\mathbb{F}$ such that $a_2-a_{-1}+\alpha(a_1-a_0)=0$.
Since
$$0=a_0(a_2-a_{-1}+\alpha(a_1-a_0))=p_{2,0}+(\alpha-1)p_1+\alpha(2\eta-1)a_0,$$
$p_{2,0}=(1-\alpha)p_1$ and either $\alpha$ or $2\eta-1$ is $0$ .

Let $q=p_1-\frac{\xi-\eta}{2}((\alpha+1)a_0+a_1+a_{-1})$.
Then $qa_0=((1-\xi)\lambda_1-\frac{\xi-\eta}{2}\alpha-\frac{\eta+\xi}{2})a_0$.
Since $f(q)=q$ and $\tau_0(q)=q$, $qa_i=((1-\xi)\lambda_1-\frac{\xi-\eta}{2}\alpha-\frac{\eta+\xi}{2})a_i$ for all $i\in\mathbb{Z}$.
By Lemma \ref{lem2}, $q^2=((1-\xi)\lambda_1-\frac{\xi-\eta}{2}\alpha-\frac{\eta+\xi}{2})q$ and then $M=\mathbb{F}q+\sum_{i=-1}^1\mathbb{F}a_i$. 
Since 
$$q-(\lambda_1-\frac{\xi-\eta}{2}\alpha-\frac{\xi+\eta}{2})a_0+\frac{\xi}{2}(a_1+a_{-1})\in \eisp{a_0}{\xi},$$
\begin{eqnarray*}
\eisp{a_0}{0}&\oplus& \eisp{a_0}{1}=\mathbb{F}q\oplus\mathbb{F}a_0\\
&\ni&(q-(\lambda_1-\frac{\xi-\eta}{2}\alpha-\frac{\xi+\eta}{2})a_0+\frac{\xi}{2}(a_1+a_{-1}))^2\\
&\in&\mathbb{F}q\oplus\mathbb{F}a_0\\
&&+\xi((1-2\xi)\lambda_1+\frac{(\xi-2)(\xi-\eta)}{4}\alpha+\frac{3\xi^2+3\xi\eta-\xi-2\eta}{4})(a_1+a_{-1}).
\end{eqnarray*}
Thus $(1-2\xi)\lambda_1+\frac{(\xi-2)(\xi-\eta)}{4}\alpha+\frac{3\xi^2+3\xi\eta-\xi-2\eta}{4}=0$ or $a_1+a_{-1}-2\lambda_1a_0\in\eisp{a_0}{0}$.

If  $a_1+a_{-1}-2\lambda_1a_0\in\eisp{a_0}{0}$, then $\eisp{a_0}{\xi}=0$ and hence $M$ is of Jordan type.

So we may assume $(1-2\xi)\lambda_1+\frac{(\xi-2)(\xi-\eta)}{4}\alpha+\frac{3\xi^2+3\xi\eta-\xi-2\eta}{4}=0$.
Then there exists a surjective homomorphism from $\threealg{\xi}{\eta}{\alpha}$ to $M$ such that $\hat{a}_i\mapsto a_i$ and $\hat{q}\mapsto q$.
Since either $\alpha$ or $2\eta-1$ is $0$, $M$ is isomorphic to $\threealg{\xi}{\eta}{0}$, $\threealg{\xi}{\frac{1}{2}}{\alpha}$ or their quotients.
\end{proof}

In the case when $M$ is isomorphic to a quotient of $\threealg{\xi}{\eta}{0}$ or  $\threealg{\xi}{\frac{1}{2}}{\alpha}$, $q=0$ since $M$ is not of Jordan type.
Therefore $M$ is isomorphic to $\threealg{\xi}{\frac{1-3\xi^2}{3\xi-1}}{0}^{\times}$. $\threealg{\xi}{\frac{1}{2}}{-3}^{\times}$ or $\threealg{-1}{\frac{1}{2}}{\alpha}^{\times}$.
Since $\eta\notin\{0,1,\xi\}$, $\xi\notin\{\pm\frac{1}{3},\pm\frac{1}{\sqrt{3}}\}$ if $M$ is isomorphic to $\threealg{\xi}{\frac{1-3\xi^2}{3\xi-1}}{0}^{\times}$ and $\chf{F}\neq3$ if $M$ is isomorphic to $\threealg{-1}{\frac{1}{2}}{\alpha}^{\times}$.
Hence the proof of Step \ref{prop3} is completed.

\subsection{Proof of Step \ref{prop4e}}
Assume that $\adim=4$ and $M$ satisfies an even relation.
In this case, there exist $\alpha,\beta\in\mathbb{F}$ such that $a_2+a_{-2}+\alpha(a_1+a_{-1})+\beta a_0=0$.
Since 
\begin{eqnarray*}
0&=&a_0(a_2+a_{-2}+\alpha(a_1+a_{-1})+\beta a_0)\\
&=&2p_{2,0}+2\alpha p_1+(\eta(2\alpha-\beta+2)+\beta)a_0,
\end{eqnarray*}
$\eta(2\alpha-\beta+2)+\beta=0$ and $p_{2,0}=-\alpha p_1$.
If $p_1\in\saxes$, then, by the $G$-invariance of $p_1$, $p_1=0$ and it contradicts Lemma \ref{lem3}.
Thus we may assume that $p_1\notin\saxes$ and $M=\mathbb{F}p_i\oplus\saxes$.
Since
\begin{eqnarray*}
0&=&p_1(a_2+a_{-2}+\alpha(a_1+a_{-1})+\beta a_0)\\
&\in&\saxes+(\xi-\eta)(2\alpha+\beta+2)p_1,
\end{eqnarray*}
$\beta=-2(\alpha+1)$ and then $0=\eta(2\alpha-\beta+2)+\beta=2(\alpha+1)(2\eta-1)$.
Since $0=a_{-1}(a_2+a_{-2}+\alpha(a_1+a_{-1})+\beta a_0)$,
$a_{-1}a_2=(\alpha+1)^2p_1+\eta a_2+((2\alpha+1)\eta-\alpha)a_{-1}$.
Then, since $a_{-1}a_2$ is invariant under the flip, $\alpha(2\eta-1)=0$.
Hence $\eta=\frac{1}{2}$.
By the $G$-invariance of $p_1$, there exists $\mu\in\mathbb{F}$ such that ${p_1}^2=\mu p_1$.
Then,
\begin{eqnarray*}
&(&p_1-(\lambda_1-\frac{1}{2})a_0+\frac{1}{4}(a_1+a_{-1}))^2\\
&&\in \eisp{a_0}{0}\oplus \eisp{a_0}{1}=\mathbb{F}a_0\oplus\mathbb{F}(p_1-\frac{2\xi-1}{4}(a_1+a_{-1})
\end{eqnarray*}
and
\begin{eqnarray*}
&(&p_1-(\lambda_1-\frac{1}{2})a_0+\frac{1}{4}(a_1+a_{-1}))^2\\
&&\in\mathbb{F}a_0+(\mu-2\xi\lambda_1+2\xi-\frac{\alpha+4}{8})p_1-\frac{2\xi-1}{4}(2\lambda_1-1+\frac{\alpha}{4})(a_1+a_{-1}).
\end{eqnarray*}
Hence $\mu=\frac{3\alpha}{8}+2(\xi+1)\lambda_1-2\xi-\frac{1}{2}$.
Then,
\begin{eqnarray*}
(&p_1&-(\lambda_1-\frac{1}{2})a_0+\frac{1}{4}(a_1+a_{-1}))(p_1-((1-\xi)\lambda_1-\frac{1}{2})a_0+\frac{2\xi-1}{4}(a_1+a_{-1}))\\
&=&(\frac{\xi+1}{4}\alpha+(\xi^2+\xi+2)\lambda_1-\xi^2-1)p_1\\
&&+(1-\xi)((\xi-1)\lambda_1^2-\frac{\xi-2}{2}\lambda_1+\frac{\xi-1}{2}+\frac{(2\xi-1)\alpha}{8})a_0\\
&&+(1-\xi)(\frac{3-4\xi}{4}\lambda_1-\frac{2\xi-1}{16}\alpha+\frac{\xi-1}{4})(a_1+a_{-1})\\
&\in&\eisp{a_0}{\xi}=\mathbb{F}(p_1-((1-\xi)\lambda_1-\frac{1}{2})a_0+\frac{2\xi-1}{4}(a_1+a_{-1})).
\end{eqnarray*}
Therefore $(\lambda_1,\alpha)=(1,-4)$ or $(\lambda_1,\xi)=(1,2)$.
Since $(\xi,\eta,\lambda_1)\neq(2,\frac{1}{2},1)$, $(\lambda_1,\alpha)=(1,-4)$.
Thus there exists a surjective homomorphism from $\fivetwo{\xi}$ to $M$ such that $\hat{a}_i\mapsto a_i$ and $\hat{p}_1\mapsto p_1$ and its kernel is $\mathbb{F}(\hat{a}_2+\hat{a}_{-2}-4(\hat{a}_1+\hat{a}_{-1})+6\hat{a}_0)$.
Hence $M$ is isomorphic to $\fivetwo{\xi}^{\times}$ and $\xi\neq\frac{1}{2}$.

\subsection{Proof of Step \ref{prop4o}}
Assume that $\adim=4$ and $M$ satisfies an odd relation.
Let $\alpha$ be an element of $\mathbb{F}$ such that $a_2-a_{-2}+\alpha(a_1-a_{-1})=0$.
\begin{claim}
\sl
\label{lem41}
One of the following holds:
\begin{itemize}
\item[\rm{(i)}]$\alpha=0$.
\item[\rm{(ii)}]$\alpha=1$ and $p_{2,0}-p_{2,1}=(2\eta-1)(a_{-1}-a_2)$.
\item[\rm{(iii)}]$2\eta=1$ and $p_{2,0}=p_{2,1}$.
\end{itemize}
\end{claim}
\begin{proof}
$a_{-1}a_2=p_1-\alpha p_{2,1}+\eta a_2+((1-2\eta)\alpha+\eta)$ in this case.
Since $a_{-1}a_2$ is invariant under the flip, 
$\alpha(p_{2,0}-p_{2,1}+(2\eta-1)(a_2-a_{-1}))=0$.
So $a_2=a_{-2}$ or $p_{2,0}-p_{2,1}=(2\eta-1)(a_{-1}-a_2)$.

Since $p_{2,1}-p_{2,0}\in\saxes$ and $p_{2,1}-p_{2,0}=\tau_0(p_{2,1}-p_{2,0})=-\theta(p_{2,1}-p_{2,0})$,
$p_{2,1}-p_{2,0}=\delta(a_2-a_{-1}+(\alpha-1)(a_1-a_0))$ for some $\delta\in\mathbb{F}$.
Hence if $p_{2,0}-p_{2,1}=(2\eta-1)(a_{-1}-a_2)$, then $\alpha=1$ or $2\eta=1$.
\end{proof}

First, we assume $\alpha=0$.
\begin{claim}
\label{lemo4}
\sl 
$\dim M\leq5$.
\end{claim}
\begin{proof}
If $\xi=2\eta$, then $\dim M\leq5$. 

Thus it suffices to verify that $(\xi,\eta,\lambda_1)=(2,\frac{1}{2},1)$ if $\xi\neq2\eta$ and $\dim M\geq6$.
By the invariance, there exist $\mu\in\mathbb{F}$ such that $p_{2,1}-p_{2,0}=\mu(a_2-a_1-a_{-1}-a_0)$.
Then, by the coefficient of $p_1$ and $p_{2,0}$ in $a_0p_{2,1}$, $\mu=\frac{-(\xi-\eta)(\xi-4\eta)}{\xi-2\eta}$ and $\lambda_1=\frac{2\xi^2+8\xi\eta-4\eta^2-\xi-2\eta}{4(2\xi-1)}$.
Thus $\xi\neq\frac{1}{2}$ and then $\lambda_2=\frac{\xi}{2}$ by the structure of $\langle a_0,a_2\rangle_{alg}$.
Set 
$$q=p_1+\frac{\xi-2\eta}{4\eta}p_{2,0}-\frac{\xi-\eta}{2}(a_1+a_{-1})-\frac{(\xi-2\eta)(\xi-\eta)}{4\eta}(a_2+a_0).$$
Then $qa_i=\pi a_i$ and $qq=\pi q$ where $\pi=(1-\xi)\lambda_1+\frac{-\xi^3-8\eta^2-\xi^2+2\xi\eta+2\xi^2\eta}{8\eta}$.
Since 
$$(q-(\lambda_1+\frac{-\xi^4+2\xi\eta-8\eta^2}{8\eta})a_0+\frac{\xi}{2}(a_1+a_{-1})+\frac{\xi(\xi-2\eta)}{4\eta}a_2)^2\in \eisp{a_0}{0}\oplus \eisp{a_0}{1},$$
$\eta=\frac{1}{2}$, $\xi=2$ and $\lambda_1=1$.
Hence $\dim M\leq5$.
\end{proof}

\begin{claim}
\sl
$\eisp{a_0}{\xi}$ is of dimension $1$.
\end{claim}
\begin{proof}
Assume $\dim \eisp{a_0}{\xi}=2$.
If $a_1+a_{-1}-2\lambda_1a_0\in\eisp{a_0}{\xi}$, then $p_1=0$ by the invariance.
Hence there exists $\mu\in\mathbb{F}$ such that $a_2+\mu(a_1+a_{-1})-(2\mu\lambda_1+\lambda_2)a_0\in \eisp{a_0}{\xi}$ and $p_1\notin\saxes$.
Then $a_0a_2=-2\mu p_1+((1-\xi)(2\mu\lambda_1+\lambda_2)-2\mu\eta)a_0+\xi a_2+\mu(\xi-\eta)(a_1+a_{-1})$
and $(1-\xi)(2\mu\lambda_1+\lambda_2)=2\mu\eta+\xi$.

Let us consider the case with $\mu=0$.
Then $\langle a_0,a_2\rangle_{alg}$ is a $2$-dimensional primitive axial algebra of Jordan type $\xi$.
So, by Lemma \ref{lemhall}, $\xi=-1$ or $\frac{1}{2}$.
In the case with $\xi=-1$, 
\begin{eqnarray*}
&&(a_2+\frac{1}{2}a_0)(p_1-(\lambda_1-\eta)a_0+\frac{\eta}{2}(a_1+a_{-1}))\\
&&=-\frac{3}{4}(2p_1-2(\lambda_1-\eta)a_0+\eta(a_1+a_{-1})+4(\lambda_1-\eta)a_2)\\
&&\in \eisp{a_0}{0}\oplus \eisp{a_0}{1}=\mathbb{F}(p_1+\frac{\eta+1}{2}(a_1+a_{-1}))\oplus\mathbb{F}a_0.
\end{eqnarray*}
Therefore $\mathbb{F}$ is of characteristic $3$ and then $-1=\frac{1}{2}$.

So we may assume  $\xi=\frac{1}{2}$.
Since 
\begin{eqnarray*}
&(&a_0a_1)(p_1-(\frac{\lambda_1}{2}-\eta)a_0+\frac{2\eta-1}{4}(a_1+a_{-1}))\\
&&=(a_0(a_1(p_1-(\frac{\lambda_1}{2}-\eta)a_0+\frac{2\eta-1}{4}(a_1+a_{-1})))
\end{eqnarray*}
by the Seress condition,
\begin{eqnarray*}
{p_1}^2=(\frac{\lambda_1}{2}+2\eta^2-2\eta)p_1&&+(\frac{1-2\eta}{4}\lambda_1+\eta^3-\frac{\eta}{4})a_0\\
&&+(\frac{1-2\eta}{8}\lambda_1+\eta^3-\frac{\eta^2}{2})(a_1+a_{-1})+(\eta^3-\eta^2+\frac{\eta}{4})a_2.
\end{eqnarray*}
By the $G$-invariance of $p_1$, $\lambda_1=2\eta$.
Therefore
\begin{eqnarray*}
(p_1-\eta a_0+\frac{\eta}{2}(a_1+a_{-1}))^2&=&-\eta p_1-\frac{2\eta^2-\eta}{4}(a_0+a_1+a_2+a_{-1})\\
&&\in \eisp{a_0}{0}\oplus \eisp{a_0}{1}.
\end{eqnarray*}
But it cannot be true because $\eta\neq0,\frac{1}{2}$.

So, we may assume $\mu\neq0$.
Then, $\langle a_0,a_2\rangle_{alg}$ is a $2$-dimensional primitive axial algebra of Jordan type $\xi$.
Set
$$q=-\frac{1}{2\mu}(a_0a_2-\xi(a_2+a_0))=p_1-\frac{\xi-\eta}{2}(a_1+a_{-1}).$$
By the structure of $\langle a_0,a_2\rangle_{alg}$, 
$$qw=-\frac{1}{2\mu}((1-\xi)\lambda_2-\xi)w=((1-\xi)\lambda_1-\eta)w$$
for all $w\in\langle a_0,a_2\rangle_{alg}$.
Thus $q^2=((1-\xi)\lambda_1-\eta)q$ and then
\begin{eqnarray*}
{p_1}^2&=&((1-\xi)\lambda_1-\eta+(\mu+2)(\xi-\eta))p_1-\frac{(\xi-\eta)^2}{2}((\xi-\eta)\mu-2\eta)(a_0+a_2)\\
&&+(\xi-\eta)(\frac{1-\xi}{2}\lambda_1+\frac{\eta}{2}-\frac{(2\xi+1)(\xi-\eta)}{4}+\eta(\xi-\eta-1))(a_1+a_{-1}).
\end{eqnarray*}
By the $G$-invariance of $p_1$, 
$$(1-\xi)\lambda_1=-(\xi-\eta)^2\mu+\xi^2-\xi\eta+\frac{\xi+\eta}{2}.$$
Since
$$(q-((1-\xi)\lambda_1-\eta)a_0)(a_2+\mu(a_1+a_{-1})-(2\mu\lambda_1+\lambda_2)a_0)\in \eisp{a_0}{\xi},$$
\begin{eqnarray*}
&q&(a_1+a_{-1})-((1-\xi)\lambda_1-\eta)(a_1+a_{-1})\\
&&=(\xi-\eta)(2(\mu+1)p_1+\frac{2\eta-2\xi-1}{2}(a_1+a_{-1})+(-(\xi-\eta)\mu+\eta)(a_0+a_2))\\
&&\in \eisp{a_0}{\xi}.
\end{eqnarray*}
So $(\eta,\mu)$ must agree with $(\frac{1}{2},-1)$.
Since 
\begin{eqnarray*}
&&(q-(\lambda_1-\frac{1}{2})a_0+\frac{\xi}{2}(a_1+a_{-1}))(a_0+a_2-a_1-a_{-1})\\
&&=(-\xi\lambda_1+2\xi^2)(a_0+a_2)+(-\xi\lambda_1+\xi^2-2\xi+\frac{1}{4})(a_1+a_{-1})\\
&&\in \eisp{a_0}{0}\oplus \eisp{a_0}{1}=\mathbb{F}a_0\oplus\mathbb{F}q,
\end{eqnarray*}
$\lambda_1=2\xi$ and $\xi^2+2\xi-\frac{1}{4}=0$.
Furthermore, since $\lambda_2=\frac{\xi}{2}$ by Lemma \ref{lemhall}, 
$\lambda_1=\frac{\xi}{4}+\frac{1}{2}$ and by the argument above, $(1-\xi)\lambda_1=2\xi^2-\xi+\frac{1}{2}$.
There exists no $\xi\in\mathbb{F}$ satisfying all of these conditions. 
Hence $M$ cannot be an axial algebra of Majorana type in this case and then the claim is proved.
\end{proof}

Thus we may assume $\dim \eisp{a_0}{\xi}=1$.
If $a_1+a_{-1}-2\lambda_1a_0\in\eisp{a_0}{\xi}$, then $p_1=0$ by the invariance. 
Hence there exists $\mu\in\mathbb{F}$ such that
$a_2+\mu(a_1+a_{-1})-(2\mu\lambda_1+\lambda_2)a_0\in \eisp{a_0}{0}$.
Then, 
$a_0a_2=-2\mu p_1+(2\mu\lambda_1+\lambda_2-2\mu\eta)a_0-\mu\eta(a_1+a_{-1}).$
By the $\tau_1$-invariance of $a_0a_2$, $2\mu\lambda_1+\lambda_2=2\mu\eta$.

\begin{claim}
\label{lemo1}
\sl
Let $\mu=0$. 
Then $M$ is isomorphic to $\fourone{\frac{1}{4}}{\eta}$, $\fourone{\xi}{\frac{\xi}{2}}$, $\fourone{\frac{1}{4}}{\frac{1}{2}}^{\times}$ or $\fourone{-\frac{1}{2}}{-\frac{1}{4}}^{\times}$.
\end{claim}
\begin{proof}
In this case, $\lambda_2=0$ and $a_0a_2=a_1a_{-1}=0$.
Set $q=p_1-\frac{\xi-\eta}{2}(a_0+a_1+a_{-1}+a_2)$.
Then, $a_iq=((1-\xi)\lambda_1-\frac{\xi+\eta}{2})a_i$
for $i=-1,0,1,2$.
So, by Lemma \ref{lem2}, 
$q^2=((1-\xi)\lambda_1-\frac{\xi+\eta}{2})q.$
Set $\pi=(1-\xi)\lambda_1-\frac{\xi+\eta}{2}$.
Since $a_2\in \eisp{a_0}{0}$ and
$$\eisp{a_0}{\xi}=\mathbb{F}(q-(\lambda_1-\frac{\xi+\eta}{2})a_0+\frac{\xi}{2}(a_1+a_{-1})+\frac{\xi-\eta}{2}a_2)$$
in this case,
\begin{eqnarray*}
&&a_2(q-(\lambda_1-\frac{\xi+\eta}{2})a_0+\frac{\xi}{2}(a_1+a_{-1})+\frac{\xi-\eta}{2}a_2)\\
&&=\xi q+\frac{\xi(\xi-\eta)}{2}a_0+\frac{\xi^2}{2}(a_1+a_{-1})+(\pi+\frac{\xi(\xi+\eta)+\xi-\eta}{2})a_2\in \eisp{a_0}{\xi}.
\end{eqnarray*}
Hence $\lambda_1=\eta$.
Thus there exists a surjective homomorphism from $\fourone{\xi}{\eta}$ to $M$ such that $\hat{a}_i\mapsto a_i$ and $\hat{q}\mapsto q$.
Furthermore, 
$$(q-(\lambda_1-\frac{\xi+\eta}{2})a_0+\frac{\xi}{2}(a_1+a_{-1})+\frac{\xi-\eta}{2}a_2)^2\in \eisp{a_0}{0}\oplus \eisp{a_0}{1}$$
and
\begin{eqnarray*}
&(&q-(\lambda_1-\frac{\xi+\eta}{2})a_0+\frac{\xi}{2}(a_1+a_{-1})+\frac{\xi-\eta}{2}a_2)^2\\
&&\in \eisp{a_0}{0}\oplus \eisp{a_0}{1}+\frac{\xi(4\xi-1)(\xi-2\eta)}{4}(a_1+a_{-1}).
\end{eqnarray*}
Since $a_1+a_{-1}-2\lambda_1a_0\notin\eisp{a_0}{0}$, $\xi=\frac{1}{4}$ or $\xi=2\eta$.

In the case when $M$ is isomorphic to $\fourone{\xi}{\eta}/I$ for some ideal $I$ of $\fourone{\xi}{\eta}$, $I\subset\eisp{a_0}{0}$ since $M$ is not of Jordan type.
Thus $I=\mathbb{F}q$ and $q\in\eisp{a_0}{0}$ since $M$ is of axial dimension 4.
Hence $M$ is isomorphic to  $\fourone{\frac{1}{4}}{\frac{1}{2}}^{\times}$ or $\fourone{-\frac{1}{2}}{-\frac{1}{4}}^{\times}$.
\end{proof}
If $M$ is isomorphic to $\fourone{\frac{1}{4}}{\eta}$ or its quotients, then $\chf{F}\neq3$ since $\frac{1}{4}\neq1$.
If $M$ is isomorphic to $\fourone{\xi}{\frac{\xi}{2}}$, then $\xi\neq2$ since $\eta\neq1$.
If $M$ is isomorphic to $\fourone{-\frac{1}{2}}{-\frac{1}{4}}$, then $\chf{F}\neq3,5$ since $\xi\neq1\neq\eta$.
Thus Step \ref{prop4o} (1), (2), (3) or (4) hold in the case when $\mu=0$. 

\begin{claim}
\label{lemo2}
\sl
Let $\mu\neq0$.
Then $M$ is isomorphic to $\fourtwo{\frac{1}{2}}{\eta}{\frac{1-4\eta}{2\eta}}$, $\fourtwo{\xi}{\frac{\xi^2}{2}}{\frac{1}{\xi}}$, $\fourtwo{\xi}{\frac{1-\xi^2}{2}}{\frac{-1}{\xi+1}}$ or  $\fourtwo{-1}{\frac{1}{2}}{-1}^{\times}$.
\end{claim}
\begin{proof}
In this case,
\begin{eqnarray*}
\eisp{a_0}{0}&\ni&(a_2+\mu(a_1+a_{-1})-2\mu\eta a_0)^2\\
&&\in \eisp{a_0}{0}+\mu(\mu-1)(-2(\mu+1)\xi+2\mu(\mu+2)\eta+1)(a_1+a_{-1})\\
&& +2\mu(1-\mu^2)(2(1-\xi)\lambda_1)+\eta(2\mu\eta-1))a_2.
\end{eqnarray*}
Thus $\mu=1$, $(\mu,\eta)=(-1,\frac{1}{2})$ or  $2(\mu+1)\xi-2\mu(\mu+2)\eta-1=2(1-\xi)\lambda_1-\eta(1-2\mu\eta)=0$.
Furthermore,
\begin{eqnarray*}
(&a_2&+\mu(a_1+a_{-1})-2\mu\eta a_0)(p_1-(\lambda_1-\eta)a_0+\frac{\eta}{2}(a_1+a_{-1}))\\
&=&(\xi+2(\xi-\eta-\xi\eta)\mu-2\mu^2\eta)p_1\\
&&+((1-\xi)\lambda_1-\mu^2\eta^2+\mu\eta(\xi-\eta)+\eta(\xi-1))a_2\\
&&+(\mu(1-\xi)\lambda_1+(\xi\eta-\eta^2-\frac{\eta}{2}-\xi\eta^2)\mu+\frac{\xi\eta}{2})(a_1+a_{-1})\\
&&+(-2\mu\eta(1-\xi)\lambda_1-\mu^2\eta^2+(-2\xi\eta^2+\eta^2+\xi\eta)\mu)a_0\\
&\in&\mathbb{F}(p_1-(\lambda_1-\eta)a_0+\frac{\eta}{2}(a_1+a_{-1}))=\eisp{a_0}{\xi}.
\end{eqnarray*}
If $\mu=1$, then $p_1\notin\saxes$.
Thus $4\xi-6\eta-1=0$ and $\lambda_1=\frac{2\eta}{3}$.

If $\mu=-1$, then $((1-\xi)\lambda_1-\frac{1}{2})(a_2-a_1-a_{-1}+a_0)\in\eisp{a_0}{\xi}$.
Hence $\lambda=\frac{1}{2(1-\xi)}$.  
Thus $2(\mu+1)\xi-2\mu(\mu+2)\eta-1=0$
and
$\lambda_1=\frac{\eta(1-2\mu\eta)}{2(1-\xi)}$ for all $\mu\neq0$.

Since $\langle a_0,a_2\rangle_{alg}$ is of Jordan type,
$w(a_0a_2-\xi(a_0+a_2))=((1-\xi)\lambda_2-\xi)w$ for all $w\in\langle a_0,a_2\rangle_{alg}$ by Lemma \ref{lemhall}.
So, 
$$(p_1+\frac{\eta}{2}(a_1+a_{-1}))(a_0a_2-\xi(a_0+a_2))=((1-\xi)\lambda_2-\xi)(p_1+\frac{\eta}{2}(a_1+a_{-1}))$$
and then
\begin{eqnarray*}
{p_1}^2&=&(2\eta^2-\eta\xi-\frac{1}{2}\eta+\frac{\xi-2\xi^2}{2\mu})p_1\\
&&+(\frac{\eta^3\mu}{2}+\eta^3-\frac{\xi\eta^2}{2}+\frac{\xi\eta-2\xi^2\eta}{4\mu})(a_0+a_1+a_{-1}+a_2).
\end{eqnarray*}
Thus there exists a surjective homomorphism from $\fourtwo{\xi}{\eta}{\mu}$ to $M$ such that $\hat{a}_i\mapsto a_i$ and $\hat{p}_1\mapsto p_1$.
Hence $M$ is isomorphic to $\fourtwo{\xi}{\eta}{\mu}$ or its quotient.
Furthermore, by Lemma \ref{lemhall}, $\lambda_2=\frac{\xi}{2}$ or $\xi=\frac{1}{2}$.

In the case with $\xi=\frac{1}{2}$, $\mu=\frac{1-4\eta}{2\eta}$.

In the case with $\lambda_2=\frac{\xi}{2}$, $(\mu,\eta)=(\frac{1}{\xi},\frac{\xi^2}{2})$ or $(\frac{-1}{\xi+1},\frac{1-\xi^2}{2})$.

In the case when $M$ is isomorphic to $\fourtwo{\xi}{\eta}{\mu}/I$ for some ideal $I$, $I\subset\eisp{a_0}{0}$ since $M$ is not of Jordan type. 
Therefore $I=\mathbb{F}((1-\mu)p_1-\frac{\xi-\eta}{2}(a_2+a_1+a_{-1}+a_0)=0)$ where $\mu\neq1$ since $M$ is of axial dimension 4.
Thus $(\xi,\eta,\mu)=(-1,\frac{1}{2},-1)$ and $M$ is isomorphic to  $\fourtwo{-1}{\frac{1}{2}}{-1}^{\times}$.
\end{proof}
If $M$ is isomorphic to $\fourtwo{\xi}{\frac{\xi^2}{2}}{\frac{1}{\xi}}$, then $\xi\notin\{0,1,2,\pm\sqrt{2}\}$.since $\eta\neq1,\xi$.
If $M$ is isomorphic to $\fourtwo{\xi}{\frac{1-\xi^2}{2}}{\frac{-1}{\xi+1}}$, then $\xi\notin\{0,\pm1,\pm\sqrt{-1},-1\pm\sqrt{2}\}$ since $\eta\notin\{0,1,\xi\}$.
If $M$ is isomorphic to $\fourtwo{-1}{\frac{1}{2}}{-1}^{\times}$, then $\chf{F}\neq3$ since $\xi\neq\eta$.
Thus Step \ref{prop4o} (5), (6), (7) or (8) hold in the case with $\mu\neq0$.

Next, we consider the case when $\alpha=1$ and $\eta\neq\frac{1}{2}$ and show that Step \ref{prop4o} (9) holds in this case.
In this case, $a_3=a_{-1}+a_0-a_2$ and then $a_3=a_{-3}$.
So $\langle a_0,a_3\rangle_{alg}$ is of Jordan type $\xi$ and $\langle a_0,a_2\rangle_{alg}$ is of Jordan type $\eta$ or isomorphic to $\threealg{\xi}{\eta}{0}$.

When $\dim M\geq6$, then $\langle a_0,a_2\rangle_{alg}$ is isomorphic to $\threealg{\xi}{\eta}{0}$ and $\xi\neq\frac{1}{2}$.
In this case, $a_0a_3\neq0$ and then $\lambda_3=\frac{\xi}{2}$.
Since $a_0(p_{2,0}-p_{2,1})=a_0(a_{-1}-a_{2})$, $(\xi,\eta)=(2,\frac{1}{2})$.

Thus we may assume $\dim M\leq5$.
Then
$$p_{2,0}=\gamma p_1+\delta a_0+\delta a_1+\frac{\delta+2\eta-1}{2}a_{-1}+\frac{\delta-2\eta+1}{2}a_2.$$
Since $a_0a_3=p_1-p_{2,0}+\eta(a_{-1}-a_2)+a_0$, $a_0a_3\neq0$ and then $\lambda_3=\frac{\xi}{2}$ or $\xi=\frac{1}{2}$.

\begin{claim}
\sl
$\eisp{a_0}{\xi}$ cannot be of dimension 2.
\end{claim}
\begin{proof}
Assume $\dim \eisp{a_0}{\xi}=2$.
Then, $\dim M=5$ and $a_2\notin \eisp{a_0}{1}\oplus \eisp{a_0}{0}\oplus \eisp{a_0}{\eta}$.
Thus $\langle a_0,a_2\rangle_{alg}$ is isomorphic to $\threealg{\xi}{\eta}{0}$.
Hence $\xi\neq\frac{1}{2}$.
Since $\dim \eisp{a_0}{0}=1$,
\begin{eqnarray*}
p_{2,0}&-&((1-\xi)\lambda_2-\eta)a_0-\frac{\xi-\eta}{2}(2a_2+a_1-a_{-1})\\
&=&\gamma p_1+(\delta-(1-\xi)\lambda_2+\eta)a_0\\
&&+(\delta-\frac{\xi-\eta}{2})a_1+\frac{\delta+\xi+\eta-1}{2}a_{-1}+\frac{\delta-2\xi+1}{2}a_2\\
&=&\gamma(p_1-((1-\xi)\lambda_1-\eta)a_0-\frac{\xi-\eta}{2}(a_1+a_{-1})).
\end{eqnarray*}
Therefore $(\gamma,\delta,\lambda_2)=(\frac{3\xi+\eta-1}{\eta-\xi},2\xi-1,\frac{3\xi+\eta-2}{\eta-\xi}\lambda_1+\frac{(\xi+\eta)(2\xi-1)}{(\eta-\xi)(\xi-1)})$.
Thus
$$z^{\prime}=a_2-\frac{2\xi-1}{\eta-\xi}a_1-\frac{\eta+\xi-1}{\eta-\xi}a_{-1}-\frac{(\xi+\eta)(2\xi-1)}{(\eta-\xi)(\xi-1)}a_0\in \eisp{a_0}{\xi}.$$
Since $z^{\prime}z^{\prime}\in \eisp{a_0}{0}\oplus \eisp{a_0}{1}$, $\varphi_0(a_3)=\lambda_1-\lambda_2+1=\frac{\xi}{2}$ and $\lambda_2=\frac{-3\xi^3-4\xi\eta+\xi+2\eta}{4(1-2\xi)}$, $(\xi,\eta)$ must be $(1,-1)$ if $\xi,\eta\neq0,\frac{1}{2}$.
Because $\xi\neq1$, it is a contradiction and then the claim is proved.
\end{proof}

Thus we may assume $\dim \eisp{a_0}{\xi}=1$.
Then
\begin{eqnarray*}
&&p_{2,0}-(\lambda_2-\eta)a_0+\frac{\eta}{2}(2a_2+a_1-a_{-1})\\
&&=\gamma p_1+(\delta-\lambda_2+\eta)a_0+(\delta+\frac{\eta}{2})a_1+\frac{\delta+\eta-1}{2}a_{-1}+\frac{\delta+1}{2}a_2\\
&&=\gamma(p_1-(\lambda_2-\eta)a_0+\frac{\eta}{2}(a_1+a_{-1})).
\end{eqnarray*}
Therefore $(\gamma,\delta,\lambda_2)=(1-\frac{2}{\eta},-1,\frac{\eta-2}{\eta}\lambda_1+1)$.
Then, $\eta a_2+a_1+(1-\eta)a_{-1}-\eta a_0\in \eisp{a_0}{0}$.
Since 
\begin{eqnarray*}
&&(\eta a_2+a_1+(1-\eta)a_{-1}-\eta a_0)^2\\
&&\in \eisp{a_0}{0}-(\xi-1)(2\eta-1)(\eta-2)(a_1+a_{-1}-\frac{2}{\eta}a_0),
\end{eqnarray*}
$\eta=2$ and then $\lambda_2=1$.
Furthermore, since 
\begin{eqnarray*}
(p_1-(\lambda_1-2)a_0+(a_1+a_{-1}))(2a_2+a_1-a_{-1}-2a_0)\\
=(1-\xi)(\lambda_1-1)(2a_2+a_1-a_{-1}-2a_0)\in \eisp{a_0}{\xi},
\end{eqnarray*}
$\lambda_1=1$.
Since $\varphi_0(a_3)=1$ and $\langle a_0,a_3\rangle_{alg}$ is of Jordan type $\xi\neq2$, $\xi=\frac{1}{2}$.
Since $\xi\neq\eta$, $\chf{F}\neq3$.
Set $q=p_1+\frac{1}{2}(2a_0+2a_1+a_{-1}+a_2)$.
Then $qa_i=0$ for $i=-1,0,1,2$. 
So, by Lemma \ref{lem2}, $q^2=0$.
Thus there exists a surjective homomorphism from $\fourthree$ to $M$ such that $\hat{a}_i\mapsto a_i$ and $\hat{q}\mapsto q$.
Hence $M$ is isomorphic to $\fourthree$ or its quotient.
If $M$ is a quotient, then $q=0$ and $M$ is isomorphic to $\fourthree^{\times}$.
Thus Step \ref{prop4o} (9) holds in this case.

Hence it suffices to verify that $\eta\neq\frac{1}{2}$ when $\alpha\neq0$.
Assume $\eta=\frac{1}{2}$ and $\alpha\neq0$.
If $\adim\leq 5$, then it follows that $p_{2,0}=p_{2,1}=\gamma p_1+\epsilon((\alpha+1)(a_0+a_1)+a_{-1}+a_2)$ for some $\gamma,\epsilon\in\mathbb{F}$.
If $\dim M=4$, then $p_{2,0}=\delta p_1$ for some $\delta\in\mathbb{F}$.
Then $0=a_0(\delta p_1-p_{2,0})$. 
But this cannot be true because $\adim=4$.

Let $\dim M\geq 6$.
Then, $\xi=2$ and since the coefficient of $p_1$ in $a_0p_{2,1}$ is $0$, $\lambda_1=1$.
Thus we may assume $\dim M\leq5$.

\begin{claim}
\sl
\label{lemo3}
$\dim \eisp{a_0}{0}\neq2$.
\end{claim}
\begin{proof}
Assume $\dim\eisp{a_0}{0}=2$.
Since $\dim \eisp{a_0}{\xi}=1$,
\begin{eqnarray*}
&&p_{2,0}-(\lambda_2-\frac{1}{2})a_0+\frac{1}{4}(a_2+a_{-2})\\
&&=\gamma p_1-(\lambda_2-(\alpha+1)\epsilon-\frac{1}{2})a_0+((\alpha+1)\epsilon+\frac{\alpha}{4})a_1+(\epsilon-\frac{\alpha}{4})a_{-1}+(\epsilon+\frac{1}{2})a_2\\
&&=\gamma(p_1-(\lambda_2-\frac{1}{2})a_0+\frac{1}{4}(a_1+a_{-1})).
\end{eqnarray*}
Therefore $(\gamma,\epsilon,\lambda_2)=(-(\alpha+2),-\frac{1}{2},-(\alpha+2)\lambda_1+1)$.
Then, $a_i((\alpha+1)(a_0+a_1)+a_{-1}+a_2)=(\alpha+2)a_i$ for $i=-1,0,1,2$.
So, by Lemma \ref{lem2}, 
\begin{eqnarray*}
0&=&((\alpha+1)(a_0+a_1)+a_{-1}+a_2)p_1-(\alpha+2)p_1\\
&=&2(\alpha+2)(\xi-1)p_1-(\lambda_1-1)(\xi-1)((\alpha+1)(a_0+a_1)+a_{-1}+a_2).
\end{eqnarray*}
Hence $\alpha=-2$ and $\lambda_1=1$ in this case, $\dim\eisp{a_0}{0}\neq2$.
\end{proof}

Thus we may assume $\dim \eisp{a_0}{0}=1$.
Then,
\begin{eqnarray*}
&&p_{2,0}-((1-\xi)\lambda_2-\frac{1}{2})a_0-\frac{2\xi-1}{4}(a_2+a_{-2})\\
&&=\gamma p_1-((1-\xi)\lambda_2-(\alpha+19\epsilon-\frac{1}{2})a_0+((\alpha+1)\epsilon-\frac{2\xi-1}{4}\alpha)a_1\\
&&+(\epsilon+\frac{2\xi-1}{4}\alpha)a_{-1}+(\epsilon-\frac{2\xi-1}{2})a_2\\
&&=\gamma(p_1-((1-\xi)\lambda_1-\frac{1}{2})a_0-\frac{2\xi-1}{4}(a_1+a_{-1}).
\end{eqnarray*}
Therefore $(\gamma,\epsilon,\lambda_2)=(-(\alpha+2),\frac{2\xi-1}{2},-(\alpha+2)\lambda_1+\frac{\xi\alpha+\xi+1}{1-\xi})$.
Then, 
$$(\xi-1)a_0+(\alpha+1)(\xi-1)a_1+(\xi-1)a_{-1}+(\xi\alpha+\xi+1)a_0\in \eisp{a_0}{\xi}.$$

If $\alpha\neq-2$, then $\dim\langle a_0,a_2\rangle_{alg}=4$ and then, by Claim \ref{lemo5}, $M$ is isomorphic to $\threealg{\xi}{\frac{1}{2}}{1-\alpha^2}$. 
Thus $\lambda_2=\frac{(2-\xi)(\alpha^2-1)+3\xi+2}{8}$ and then $\lambda_1=\frac{(\xi-2)\alpha}{8}+\frac{\xi^2+\xi+2}{4(1-\xi)}$.
Set
\begin{eqnarray*}
q&=&\frac{1}{\alpha+2}(p_2-\frac{2\xi-1}{4}((2-\alpha^2)a_0+a_2+a_{-2})\\
&=&p_1-\frac{2\xi-1}{4}(\alpha a_0+a_1+a_{-1}).
\end{eqnarray*}
Then, by the structure of $\threealg{\xi}{\frac{1}{2}}{1-\alpha^2}$, $qw=\frac{-\xi(\xi+1)(\alpha-2)}{8}w$ for all $w\in\langle a_0,a_2\rangle_{alg}$.
Since 
\begin{eqnarray*}
&(&\xi-1)a_0+(\alpha+1)(\xi-1)a_1+(\xi-1)a_{-1}+(\xi\alpha+\xi+1)a_0\\
&&\in\langle a_0,a_2\rangle_{alg}+\frac{(\alpha+2)(\xi-1)}{2}(a_1+a_{-1}),
\end{eqnarray*}
\begin{eqnarray*}
\eisp{a_0}{\xi}&\ni&q(a_1+a_{-1})+\frac{\xi(\xi+1)(\alpha-2)}{8}(a_1+a_{-1})\\
&=&2(2\xi-1)p_1+\frac{2\xi-1}{2}((-\xi\alpha-\xi+1)a_0\\
&&+((1-\xi)\alpha-\xi)a_1-\xi a_{-1}+(1-\xi)a_2).
\end{eqnarray*}
But this cannot be true.

Hence we may assume $\alpha=-2$.
Then, $a_i(a_2+a_{-1}-a_1-a_0)=\xi(a_2+a_{-1}-a_1-a_0)$ for $i=-1,0,1,2$ and $p_1(a_2+a_{-1}-a_1-a_0)=(1-\xi)(\lambda_1-1)(a_2+a_{-1}-a_1-a_0)$.
So $\mathbb{F}(a_2+a_{-1}-a_{1}-a_0)$ is an ideal of $M$.
Then, by Claim \ref{lemo5}, $M/\mathbb{F}(a_2+a_{-1}-a_1-a_0)$ is a quotient of $\infalg$.

Hence the proof of Step \ref{prop4o}  is completed.

\subsection{Proof of Step \ref{prop5}}
Assume $\adim=5$.
\begin{claim}
\sl
\label{lem31}
$\dim M\geq6$.
\end{claim}
\begin{proof}
Assume $\dim M=5$.
It suffices to verify that $p_1=0$ since it contradicts Lemma \ref{lem3}.
Since $\dim M\leq5$, $M=\saxes$.
By Lemma \ref{lem1} and the $G$-invariance of $0$, $a_3+a_{-2}+\alpha(a_2+a_{-1})+\beta(a_1+a_0)=0$ or $a_3-a_{-2}+\alpha(a_2-a_{-1})+\beta(a_1-a_0)=0$ for some $\alpha,\beta\in\mathbb{F}$.

If $M$ satisfies an even relation, then there exist $\alpha,\beta\in\mathbb{F}$ such that $a_3+a_{-2}+\alpha(a_2+a_{-1})+\beta(a_1+a_0)=0$.
By the $G$-invariance of $p_1$, $p_1=0$.

Otherwise, there exist $\alpha,\beta\in\mathbb{F}$ such that $a_3-a_{-2}+\alpha(a_2-a_{-1})+\beta(a_1-a_0)=0$.
Since $\tau_0(p_1)=p_1$, $f(p_1)=p_1$, $\tau_0(p_{2,0})=p_{2,0}$, $\tau_0(p_{2,1})=p_{2,1}$ and $f(p_{2,0})=p_{2,1}$,
there exist $x,y\in\mathbb{F}$ such that $p_1=\gamma(a_2+a_{-2}+(\alpha+1)(a_1+a_{-1})+(\alpha+\beta+1)a_0)$ and $p_{2,0}=p_{2,1}=\delta(a_2+a_{-2}+(\alpha+1)(a_1+a_{-1})+(\alpha+\beta+1)a_0)$.
Then, 
$$0=a_0(\delta p_1+\gamma p_{2,0})=\epsilon a_0+\frac{\delta\eta(\xi-\eta)}{2}(a_1+a_{-1})+\frac{\gamma\eta(\xi-\eta)}{2}(a_2+a_{-2}),$$
where $\epsilon=(1-\xi)(\delta\lambda_1-\gamma\lambda_2)+\eta(\xi-\eta-1)(\delta-\gamma)\in\mathbb{F}$.
Therefore $\gamma=\delta=0$ and then $p_1=0$.
Hence the lemma is proved.
\end{proof}

Hence it suffices to consider the case with $\dim M\geq6$
\begin{claim}\label{lem5}
\sl
$M$ satisfies an odd relation.
\end{claim}
\begin{proof}
Assume that $M$ satisfies an even relation.
There exist $\alpha,\beta\in\mathbb{F}$ such that $a_3+a_{-2}+\alpha(a_2+a_{-1})+\beta(a_1+a_0)=0$.
Then $a_{-1}a_2=-(\alpha+\beta)p_1-(\alpha+1)p_{2,1}+\eta a_2-((2\alpha+1)\eta+\beta)a_{-1}.$
By the flip-invariance, $(\alpha+1)(p_{2,0}-p_{2,1})+(2(\alpha+1)\eta+\beta)(a_2-a_{-1})=0$.
By the $\tau_0$-invariance, $(\alpha,\beta)=(-1,0)$ or $p_{2,1}=p_{2,0}$.

Assume $(\alpha,\beta)=(-1,0)$.
Since $a_2a_{-2}=(1-\eta)a_2+\eta a_{-2}$, $\eta=\frac{1}{2}$.
By considering the structure of $\langle a_0,a_2\rangle_{alg}$, $\lambda_2=1$ and $\xi=2$.
If $\dim M=6$, then $\dim \eisp{a_0}{0}$ or $\dim \eisp{a_0}{\xi}$ is 1 and  there exist $\gamma,\delta\in\mathbb{F}$ such that $p_{2,0}=\gamma p_1+\delta(a_2+a_{-2}-2(a_1+a_{-1})+2a_0)$.
Then $\lambda_1=1$ in both cases because 
$$\gamma(p_1-(\lambda_1-\frac{1}{2})a_0+\frac{1}{4}(a_1+a_{-1}))=p_{2,0}-\frac{1}{2}a_0+\frac{1}{4}(a_2+a_-2)$$
or
$$\gamma(p_1+(\lambda_1+\frac{1}{2})a_0-\frac{3}{4}(a_1+a_{-1}))=p_{2,0}+\frac{3}{2}a_0+\frac{3}{4}(a_1+a_{-1}).$$
If $\dim M\geq7$, then there exist $\delta\in\mathbb{F}$ such that $p_{2,1}=p_{2,0}+\delta(a_2+a_{-2}-2(a_1+a_{-1})+2a_0)$.
Then by considering $a_0p_{2,1}$, $\lambda_1=1$.
So this case is not appropriate.

Assume $p_{2,0}=p_{2,1}$.
If $\dim M=6$, then $p_{2,0}\in\mathbb{F}p_1$ by the invariance  and then $\adim\leq4$.
So $\dim M\geq7$.
Then by the coefficients of $a_0p_{2,1}$, $\xi=4\eta$ and $\lambda_1=\frac{30\eta^3-4\eta}{2(8\eta-1)}$.
By considering the structure of $\langle a_0,a_2\rangle_{alg}$, $\lambda_2=1$ and $\xi=2$.
Thus $\lambda_1=1$ and this case is not appropriate.
Hence $M$ satisfies an odd relation.
\end{proof}

Thus we may assume that $M$ satisfies an odd relation.
Then there exist $\alpha,\beta\in\mathbb{F}$ such that $a_3-a_{-2}+\alpha(a_2-a_{-1})+\beta(a_1-a_0)=0$.
Then $a_2a_{-1}=(\alpha-\beta)p_1+(1-\alpha)p_{2,1}+\eta a_2+((1-2\eta)\beta+\eta)a_{-1}$.
By the flip-invariance, $(1-\alpha)(p_{2,0}-p_{2,1})=(2\eta-1)\beta(a_2-a_{-1})$.
By the $\tau_0$-invariance, $(2\eta-1)\beta=0$.
By the $\tau_0$-invariance of 
\begin{eqnarray*}
a_2a_{-2}&=&(\alpha\beta-\alpha^2+\beta+1)p_1+\alpha(\alpha-1)p_{2,1}-\beta p_{2,0}\\
&&+((1-2\eta)\alpha+\eta)a_2+\eta a_{-2}-(1-2\eta)\alpha\beta a_{-1},
\end{eqnarray*}
$\eta=\frac{1}{2}$ or $\alpha=\beta=0$.
By a similar argument as Lemma \ref{lem31}, $p_{2,0}-p_{2,1}=0$ from its invariance.

\begin{claim}
\label{lem51}
\sl
Let $\alpha=\beta=0$.
Then $M$ is isomorphic to $\fiveone{\xi}$.
\end{claim}
\begin{proof}
If $\dim M\geq 7$, then $p_{2,0}\notin\mathbb{F}p_1+\sum_{i\in\mathbb{Z}}\mathbb{F}a_i$.
Since $a_0(p_{2,1}-p_{2,0})=0$, $\xi=4\eta$ and $\lambda_1=\frac{30\eta^3-4\eta}{2(8\eta-1)}$.
Set $q=p_1+p_{2,0}-\frac{3\eta}{2}(a_2+a_{-2}+a_1+a_{-1}+a_0)$.
Then for all $w\in\mathbb{F}p_{2,0}+\mathbb{F}p_1+\sum_{i\in\mathbb{Z}}\mathbb{F}a_i$, $qw=((1-4\eta)(\lambda_1+\lambda_2)-\frac{7\eta}{2})w$.
Since 
$$(q-(\lambda_1+\lambda_2-\frac{7\eta}{2})a_0+2\eta(a_2+a_{-2}+a_1+a_{-1}))^2\in \eisp{a_0}{0}\oplus \eisp{a_0}{1}$$
and
\begin{eqnarray*}
&(&q-(\lambda_1+\lambda_2-\frac{7\eta}{2})a_0+2\eta(a_2+a_{-2}+a_1+a_{-1}))\\
&&\cdot(p_1-(\lambda_1-\eta)a_0+\frac{\eta}{2}(a_1+a_{-1}))\\
&&\in \eisp{a_0}{0}\oplus \eisp{a_0}{1},
\end{eqnarray*}
$\lambda_2=\frac{14\eta^2-\eta}{8\eta-1}$ and $\eta=\frac{1}{2}$.
Thus this case does not hold.

So we may assume $\dim M=6$.
Then there exist $\gamma,\delta\in\mathbb{F}$ such that $p_{2,0}=\gamma p_1+\delta(a_2+a_{-2}+a_1+a_{-1}+a_0)$ and either $\dim \eisp{a_0}{0}$ or $\dim \eisp{a_0}{\xi}$ is 1.
If $\dim \eisp{a_0}{0}=1$, then
$$\gamma(p_1-((1-\xi)\lambda_1-\eta)a_0-\frac{\xi-\eta}{2}(a_1+a_{-1}))=p_{2,0}-((1-\xi)\lambda_2-\eta)a_0-\frac{\xi-\eta}{2}(a_2+a_{-2}).$$
Therefore $\lambda_1+\lambda_2=\frac{\xi+3\eta}{2(1-\xi)}$, $(\gamma,\delta)=(-1,\frac{\xi-\eta}{2})$ and $(1-\xi)(a_2+a_{-2}+a_1+a_{-1})-(\xi+3\eta)a_0\in \eisp{a_0}{\xi}$.
Since  
\begin{eqnarray*}
&(&(1-\xi)(a_2+a_{-2}+a_1+a_{-1})-(\xi+3\eta)a_0)\\
&&\cdot(p_1-((1-\xi)\lambda_1-\eta)a_0-\frac{\xi-\eta}{2}(a_1+a_{-1}))\\
&&\in \eisp{a_0}{\xi}
\end{eqnarray*}
and 
$$((1-\xi)(a_2+a_{-2}+a_1+a_{-1})-(\xi+3\eta)a_0)^2\in M_{0}(a_0)\oplus \eisp{a_0}{1},$$
$\mathrm{ch}\mathbb{F}=5$ and $\eta=\frac{1}{2}$.

Thus we may assume $\dim \eisp{a_0}{\xi}=1$.
Then 
$$\gamma(p_1-(\lambda_1-\eta)a_0+\frac{\eta}{2}(a_1+a_{-1}))=p_{2,0}-(\lambda_2-\eta)a_0+\frac{\eta}{2}(a_2+a_{-2}).$$
Therefore $\lambda_1+\lambda_2=\frac{3\eta}{2}$, $(\gamma,\delta)=(-1,\frac{-\eta}{2})$ and $(a_2+a_{-2}+a_1+a_{-1})-3\eta a_0\in M_{0}(a_0)$.
Thus $a_i(a_2+a_{-2}+a_1+a_{-1}+a_0)=(3\eta+1)a_i$ for all $i$ and then $p_1(a_2+a_{-2}+a_1+a_{-1}+a_0)=(3\eta+1)p_1$.
Therefore $\eta=\frac{5\xi-1}{8}$ and $\lambda_1=\lambda_2=\frac{3\eta}{4}$.
Then 
$$p_1^2=\frac{-5(3\xi+1)(5\xi-1)}{128}p_1-\frac{(7\xi-3)(3\xi+1)(5\xi-1)}{2048}(a_2+a_{-2}+a_1+a_{-1}+a_0)$$
and $M$ is isomorphic to $\fiveone{\xi}$.
\end{proof}
Since $\eta\notin\{0,1,\xi\}$, $\xi\notin\{\frac{1}{5},\frac{9}{5},\frac{-1}{3}\}$ in this case.

\begin{claim}
\label{lem52}
\sl
Let $\eta=\frac{1}{2}$.
Then $M$ is isomorphic to $\fivetwo{\xi}$.
\end{claim}
\begin{proof}
If $\dim M\geq7$, then $\xi=2$ and $\lambda_1=1$ by the same argument as above.
So we may assume $\dim M=6$.
Then there exist $\gamma,\delta\in\mathbb{F}$ such that $p_{2,0}=\gamma p_1+\delta(a_2+a_{-2}+(\alpha+1)(a_1+a_{-1})+(\alpha+\beta+1)a_0)$ and either $\dim \eisp{a_0}{0}$ or $\dim \eisp{a_0}{\xi}$ is 1.

Assume $\dim \eisp{a_0}{\xi}=1$.
Since 
$$\gamma(p_1-(\lambda_1-\eta)a_0+\frac{\eta}{2}(a_1+a_{-1}))=p_{2,0}-(\lambda_2-\eta)a_0+\frac{\eta}{2}(a_2+a_{-2}),$$
$(\gamma,\delta)=(-(\alpha+1),\frac{1}{4})$ and $(\alpha+1)\lambda_1+\lambda_2=\frac{\alpha-\beta+3}{4}$.
Hence $a_0(a_2+a_{-2}+(\alpha+1)(a_1+a_{-1})+(\alpha+\beta+1)a_0)=\frac{3\alpha+\beta+5}{2}a_0$.
Thus $p_1(a_2+a_{-2}+(\alpha+1)(a_1+a_{-1})+(\alpha+\beta+1)a_0)=\frac{3\alpha+\beta+5}{2}p_1$ and $3\alpha+\beta+5=0$.
By considering the quotient, $\alpha=-5$ or $\xi=2$.
If $\xi=2$, $M$ is a quotient of $\infalg$.
If $\alpha=-5$, then $p_1^2=\frac{(2\xi-1)(2\xi-3)}{32}(a_2+a_{-2}-4(a_{1}+a_{-1})+6a_0)$.
Thus there exists a surjective homomorphism from $\fivetwo{\xi}$ to $M$ such that $\hat{a}_i\mapsto a_i$ and $\hat{p}_i\mapsto p_1$.
Since there exist no appropriate ideal, the surjection is an isomorphism.

Assume $\dim \eisp{a_0}{0}=1$.
Since
$$\gamma(p_1-((1-\xi)\lambda_1-\eta)a_0-\frac{\xi-\eta}{2}(a_1+a_{-1}))=p_{2,0}-((1-\xi)\lambda_2-\eta)a_0-\frac{\xi-\eta}{2}(a_2+a_{-2}),$$
$(\gamma.\delta)=(-(\alpha+1),\frac{2\xi-1}{4})$ and $(\alpha+1)\lambda_1+\lambda_2=\frac{(2\xi+1)\alpha+(2\xi-1)\beta+2\xi+3}{4(1-\xi)}$.
Since 
\begin{eqnarray*}
&(&a_2+a_{-2}+(\alpha+1)(a_1+a_{-1})-((\alpha+1)\lambda_1+\lambda_2)a_0)\\
&&\cdot(p_1-((1-\xi)\lambda_1-\frac{1}{2})a_0-\frac{2\xi-1}{4}(a_1+a_{-1}))\\
&&\in \eisp{a_0}{\xi},
\end{eqnarray*}
$\beta=-3\alpha-5$.
Then $\mathbb{F}(a_2+a_{-2}+(\alpha+1)(a_1+a_{-1})+(\alpha+\beta+1)a_0)$ is an ideal.
By the structure of the quotient, $\alpha=-5$ or $\xi=2$.
If $\xi=2$, then $M$ is a quotient of $\infalg$.
If $\alpha=-5$, then $p_1^2=\frac{3-6\xi}{32}(a_2+a_{-2}-4(a_1+a_{-1})+6a_0)$.
Since $(p_1+\frac{2\xi-1}{2}a_0-\frac{2\xi-1}{4}(a_1+a_{-1}))^2\in \eisp{a_0}{0}$, $\xi=2$.
\end{proof}

Thus, the proof of Step \ref{prop5} is completed.

\subsection{Proof of Step \ref{prop6}}
Assume $\adim\ge6$ and $\chf{F}\neq5$.
\begin{claim}
\label{lem64}
\sl
Let $\xi=4\eta$.
Then $M$ is isomorphic to $\rmsix_2(\frac{1}{3},\frac{1}{12})$.
\end{claim}
\begin{proof}
In this case, $p_{2,1}=p_{2,0}$ and thus $\lambda_1=\frac{18\eta-1}{8}$.
Since $a_0(p_{2,1}-p_{2,0})=0$, $0\in\mathbb{F}a_0+(2(1-8\eta)\lambda_1+30\eta^3-4\eta)(2p_1+\eta(a_1+a_{-1}))$.
Thus $\eta=\frac{1}{2}$ or $\eta=\frac{1}{12}$.
By the assumption, $\eta=\frac{1}{12}$ and then $\chf{F}\neq5$.
Then $a_0(p_{3,1}+p_{3,-1})=\frac{1}{2}p_{3,0}+\frac{1}{48}(a_3+a_{-3})+\frac{5}{6}(p_1+p_{2,0}+\frac{1}{24}(a_2+a_{-2}+a_1+a_{-1}))$. By the invariance of $p_1p_{2,0}$, $a_4-a_{-3}+a_3-a_{-2}+\alpha(a_2-a_{-1})=0$ for some $\alpha\in\mathbb{F}$ and $p_{3,1}=p_{3,0}-\frac{1}{12}(a_3-a_{-2}+(\alpha-1)(a_1-a_0))$.
Therefore $p_{3,0}=-4p_{2,0}-(\alpha+4)p_1-\frac{1}{24}(a_3+a_{-3})-\frac{1}{6}(a_2+a_{-2})-\frac{\alpha+4}{24}(a_1+a_{-1}+a_0)$.
Then 
\begin{eqnarray*}
a_2a_{-3}&=&(\alpha^2+3\alpha+1)p_1+5\alpha p_{2,0}\\
&&+\frac{3\alpha+2}{24}a_3+\frac{\alpha}{24}a_{-3}+\frac{\alpha}{8}a_2+\frac{-\alpha^2+12\alpha+1}{12}a_{-2}\\
&&+\frac{3\alpha^2+2\alpha}{24}a_1+\frac{\alpha^2+4\alpha}{24}(a_{-1}+a_0).
\end{eqnarray*}
By the flip-invariance of $a_3a_{-2}$, $\frac{\alpha^2-10\alpha}{12}(a_3-a_{-2})+\frac{\alpha^2-\alpha}{12}(a_1-a_0)=0$.
Since $\adim\geq6$, $\alpha=0$.
Then $a_3a_{-3}=\frac{11}{12}a_3+\frac{1}{12}a_{-3}$.
By the $\tau_0$-invariance, $a_3=a_{-3}$ since $\chf{F}\neq5$.
Then there exists a homomorphism from $\rmsix_2(\frac{1}{3},\frac{1}{12})$ to $M$ such that $\hat{a}_i\mapsto a_i$.
Since there does not exists an appropriate ideal, this homomorphism is an isomorphism.
\end{proof}

Thus we may assume $\xi\neq4\eta$.
\begin{claim}\label{lemst6}
\sl
There exist $\alpha,\beta\in\mathbb{F}$ such that $a_3-a_{-3}+\alpha(a_2-a_{-2})+\beta(a_1-a_{-1})=0$ and one of the following conditions holds.
\begin{itemize}
\item[\rm{(i)}]$\alpha=\beta=0$
\item[\rm{(ii)}]$\eta=\frac{1}{2}$, $p_{3,0}=p_{3,1}$ and $\alpha(\beta-1)=0$.
\end{itemize}
\end{claim}
\begin{proof}
Since $p_{2,1}-p_{2,0}\in\sum_{i=-1}^{2}\mathbb{F}a_i+\frac{\xi-4\eta}{4}(a_3-a_{-2})$ by (\ref{p2linrel}) and $p_{2,0}-p_{2,1}$ is $\tau_0$-invariant, $\adim\leq6$ and there exist $\alpha,\beta\in\mathbb{F}$ such that $a_3-a_{-3}+\alpha(a_2-a_{-2})+\beta(a_1-a_{-1})=0$.
Then $a_2a_{-2}=-\alpha p_{3,-1}+(1-\beta)p_{2,0}+\alpha p_1+\eta a_{-2}+(\eta+(1-2\eta)\beta)a_2$.
By the $\tau_0$-invariance, $\alpha(p_{3,1}-p_{3,-1})=(2\eta-1)\beta(a_2-a_{-2})$.
Thus $(2\eta-1)\beta(a_3-a_{-3}+a_2-a_{-2}+a_1-a_{-1})=0$ and then $\alpha=\beta=1$, $\beta=0$ or $\eta=\frac{1}{2}$.
Furthermore,
\begin{eqnarray*}
a_{-2}a_3=&&\alpha^2p_{3,-1}-\beta p_{3,1}+\alpha(\beta-1)p_{2,0}+(-\alpha^2+\beta+1)p_1\\
&&+\eta a_3+(\eta+(1-2\eta)\alpha)a_{-2}+(1-2\eta)\alpha\beta a_2.
\end{eqnarray*}
By the flip-invariance, 
$$\beta(p_{3,0}-p_{3,1})+\alpha(\beta-1)(p_{2,0}-p_{2,1})+(1-2\eta)\alpha\beta(a_{-1}-a_2)+(1-2\eta)\alpha(-a_3+ a_{-2})=0.$$
If $\alpha=0$ and $\beta\neq0$, then $\eta=\frac{1}{2}$ and $p_{3,0}=p_{3,1}$.
If $\beta=0$, then $\alpha=1$. 
In this case, $a_5=a_{-5}$ and then by considering $\langle a_0.a_2\rangle_{alg}$ and the invariance of $p_{3,1}$, this case is not appropriate.
If $\alpha\neq0\neq\beta$, then 
By the flip-invariance, $\beta=1$ and $2\eta=1$.
Hence $p_{3,0}=p_{3,1}$.
Thus if $(\alpha,\beta)\neq(0,0)$, then $\eta=\frac{1}{2}$, $p_{3,0}=p_{3,1}$ and $\alpha(\beta-1)=0$.
\end{proof}

\begin{claim}
\label{lem61}
\sl
Let $\xi=2\eta$.
Then $M$ is isomorphic to $\sixone{\xi}$ or $\rmsix_1(\frac{-2}{7},\frac{-1}{7})^{\times}$.
\end{claim}
\begin{proof}
In this case, $\eta\neq\frac{1}{2}$ and then $a_3=a_{-3}$.
Since
$$p_{2,1}\in\frac{2(4\eta-1)(\lambda_1-\eta)}{\eta}p_1+(4\eta-1)(\lambda_1-\eta)(a_0+a_2)-\frac{\eta}{2}(a_3+a_{-1})+\mathbb{F}a_0,$$
$\lambda_1=\eta$ or $\eta=\frac{1}{4}$ by the $\tau_0$-invariance.
In both cases, $p_{2,1}=-\frac{\eta}{2}(a_3+a_1+a_{-1})$.
Then
\begin{eqnarray*}
p_1p_1&=&\frac{1}{2\eta}a_0(z_1x_1-z_1z_1)-(z_1x_1-p_1p_1)\\
&\in&\frac{\eta^2}{4}p_{3,0}+\frac{2(1-2\eta)\lambda_1+3\eta^3-4\eta}{2}p_1+\frac{2\eta(1-2\eta)(\lambda_1-\eta)}{4}(a_1+a_{-1})\\
&&+\mathbb{F}a_0.
\end{eqnarray*}
Since $f_3(p_1p_1)=p_1p_1$, $\lambda_1=\eta$ and $p_1p_1=\frac{\eta^2}{4}p_{3,0}-\frac{\eta^2+\eta}{2}p_1$.
Therfore $p_{3,0}=p_{3,1}$.
Set $p_3=p_{3,0}$ and
$$q=p_3+2p_1-\eta(a_3+a_2+a_{-2}+a_1+a_{-1}+a_0).$$
Then by Lemma \ref{lem2}, $\eisp{q}{-(7\eta^2+\eta)}\supset\mathbb{F}p_1+\mathbb{F}p_3+\sum_{i\in\mathbb{Z}}\mathbb{F}a_i$.
Thus $M=\mathbb{F}p_1+\mathbb{F}p_3+\sum_{i\in\mathbb{Z}}\mathbb{F}a_i$ and there exists a surjective homomorphism from  $\sixone{\xi}$ to $M$ such that $\hat{a}_i\mapsto a_i$, $\hat{p}_1\mapsto p_1$ and $\hat{q}\mapsto q$.

When $M$ is isomorphic to $\sixone{\xi}/I$ for an ideal $I\neq\{0\}$, all elements of $I$ are $G$-invariant because $M$ is of axial dimension 6.
Thus $I=\mathbb{F}q$ and $qw=0$ for all $w\in M$.
Hence $M$ is isomorphic to $\rmsix_1(\frac{-2}{7},\frac{-1}{7})^{\times}$.
\end{proof} 
Since $\eta\neq1$, $\xi\neq2$ in this case.
Furthermore, the quotient  $\rmsix_1(\frac{-2}{7},\frac{-1}{7})^{\times}$ exists only if $\chf{F}\neq7$.

Thus we may assume $\xi\neq2\eta$ and $\xi\neq4\eta$.
\begin{claim}
\label{lem62}
\sl
$M$ is isomorphic to $\sixtwo{\xi}$, $\rmsix_2(\frac{2}{3},\frac{-1}{3})^{\times}$ or $\rmsix_2(\frac{1\pm\sqrt{97}}{24},\frac{53\pm5\sqrt{97}}{192})^{\times}$.
\end{claim}

\begin{claim}
\label{lem621}
\sl
$\alpha=\beta=0$.
\end{claim}
\begin{proof}
Assume $(\alpha,\beta)\neq(0,0)$.
By Claim \ref{lemst6}, $\alpha(\beta-1)=0$, $\eta=\frac{1}{2}$ and $p_{3,1}=p_{3,0}$.
If $\alpha=0$,  $p_{2,1}-p_{2,0}=\frac{\xi-2}{4}(a_3-a_{-2}-a_{2}+a_{-1}+(\beta+1)(a_1-a_0))$.
Therefore 
$$0=p_{3,0}=p_{3,1}\in\frac{(\xi-1)}{2}(a_3-a_{-2})+\sum_{-1}^2\mathbb{F}a_i.$$
Thus we may assume $\beta=1$.
Since $p_{2,1}=p_{2,0}+\frac{\xi-2}{4}(a_3-a_{-2}+(\alpha-1)(a_2-a_{-1})+(2-\alpha)(a_1-a_0)$ by the invariance,
$\lambda_1=\frac{2\xi^2-9\xi+5}{4(\xi^2-4\xi+2)}$.
Furthermore, since $p_{3,0}=p_{3,1}$,
\begin{eqnarray*}
p_{3,0}&=&a_0(\frac{4}{\xi-2}(p_{2,1}-p_{2,0})+a_{-2}-(\alpha-1)(a_2-a_{-1})+(\alpha-2)(a_1-a_0))\\
&&-\eta(a_0+a_3)\\
&\in&\mathbb{F}p_1+\mathbb{F}(p_{2,0}+p_{2,1})+\mathbb{F}(a_3+a_{-2}+(\alpha+1)(a_2+a_{-1})+(\alpha+2)(a_1+a_0)).
\end{eqnarray*}
Thus $(\alpha,\xi)=(-4,3)$.
Then $\langle a_0,a_2\rangle_{alg}$ cannot be a dihedral $(3,\frac{1}{2})$-algebra.
\end{proof}

\begin{proof}[Proof of Claim \ref{lem62}]
By Claim \ref{lem621}, we may assume $a_3=a_{-3}$.
Since $p_{2,1}-p_{2,0}=\frac{\xi-4\eta}{4}(a_3-a_{-2}-a_2+a_{-1}+a_1-a_0)$ by the invariance,
$(2\xi^2-12\xi\eta-2\xi+8\eta)\lambda_1=4\eta^3-10\xi\eta^2-\xi\eta+6\eta^2$.
Since 
\begin{eqnarray*}
p_{3,0}&\in&\mathbb{F}a_0+\mathbb{F}p_1+\mathbb{F}(p_{2,0}+p_{2,1})+\mathbb{F}a_3-\frac{\xi^2}{4(\xi-4\eta)}(a_2+a_{-2})\\
&&+\frac{32\eta(1-2\xi)\lambda_1+\xi^2+20\xi\eta^2+24\xi\eta^2-8\xi\eta-16\eta^2}{4(\xi-2\eta)(\xi-4\eta)}(a_1+a_{-1})
\end{eqnarray*}
and $p_{3,0}\in\mathbb{F}p_1+\mathbb{F}(p_{2,1}+p_{2,0})+\mathbb{F}(a_0+a_3)+\mathbb{F}(a_1+a_{-1}+a_{2}+a_{-2})$ by the invariance,
$$16\eta(1-2\xi)\lambda_1=-\xi^3-8\xi^2\eta-12\xi\eta^2+4\xi\eta+8\eta^2.$$
Therefore $\eta=\frac{\xi-1}{2}$ or $\eta=\frac{-\xi^2}{4(2\xi-1)}$

If $\eta=\frac{-\xi^2}{4(2\xi-1)}$, then
\begin{eqnarray*}
p_{3,i}&=&\frac{2(2\xi-1)}{\xi}(2p_1-\frac{\xi^2}{4(2\xi-1)}(a_{i+1}+a_{i-1}))\\
&&-\frac{2(2\xi-1)}{5\xi-2}(2p_{2,i}-\frac{\xi^2}{4(2\xi-1)}(a_{i+2}+a_{i-2}))\\
&&+\frac{\xi^2}{4(2\xi-1)}a_{i+3}+\frac{\xi(29\xi^2-22\xi+4)}{4(2\xi-1)(5\xi-2)}a_i,
\end{eqnarray*}
\begin{eqnarray*}
p_1p_1&=&\frac{\xi^2(39\xi^2-22\xi+2)}{16(2\xi-1)^2}p_1-\frac{\xi^4(9\xi-4)}{32(2\xi-1)^2(5\xi-2)}p_{2,0}\\
&&-\frac{\xi^4(3\xi-1)(9\xi-4)}{128(2\xi-1)^3}(a_3+a_1+a_{-1})\\
&&-\frac{\xi^4(3\xi-1)(9\xi-4)}{64(2\xi-1)^2(5\xi-2)}(a_2+a_{-2}+a_0).
\end{eqnarray*}
Set
\begin{eqnarray*}
q&=&p_{2,0}+\frac{(3\xi-2)(5\xi-2)^2}{\xi^2(9\xi-4)}p_1\\
&&-\frac{(3\xi-2)(5\xi-2)}{8(2\xi-1)}(a_{3}+a_{1}+a_{-1})-\frac{21\xi^2-18\xi+4}{8(2\xi-1)}(p_{2}+p_{-2}+p_0).
\end{eqnarray*}
By Lemma \ref{lem2}, $\{p_1\}\cup\{p_{2,0}\}\cup\{a_i\mid i\in\mathbb{Z}\}\subset\eisp{q}{-\frac{(3\xi-2)(5\xi-2)(12\xi^2-\xi-1)}{8(2\xi-1)(9\xi-4)}}$.
Thus $M=\mathbb{F}p_1+\mathbb{F}q+\sum_{i=-2}^3\mathbb{F}a_i$ and then the map $\sixtwo{\xi}$ to $M$ such that $\hat{a}_i\mapsto a_i$, $\hat{p}_1\mapsto p_1$ and $\hat{q}\mapsto q$ is a surjective homomorphism.

If $M$ is not isomorphic to $\sixtwo{\xi}$, then $M=\sixtwo{\xi}/\mathbb{F}q$ and $qw=0$ for all $w\in M$ by a similar argument as Claim \ref{lem61}. 
Hence $M$ is isomorphic to $\rmsix_2(\frac{2}{3},\frac{-1}{3})^{\times}$ or $\rmsix_2(\frac{1\pm\sqrt{97}}{24},\frac{53\pm5\sqrt{97}}{192})^{\times}$.

If $\eta\neq\frac{-\xi^2}{4(2\xi-1)}$, then $\eta=\frac{\xi-1}{2}$, $\lambda_3=\frac{\xi}{2}$ and $\lambda_2=\frac{9\xi^2-7\xi+2}{8(2\xi-1)}$ by the structure of $\langle a_0,a_2\rangle_{alg}$ and $\langle a_0,a_3\rangle_{alg}$.
Since $x_3\in\mathbb{F}x_1+\mathbb{F}x_2$,
$\lambda_3-\eta+\frac{2\xi}{\xi-1}(\lambda_1-\eta)+2\xi(\lambda_2-\eta)=\frac{-\xi^2+\xi+1}{2}.$
Therefore $(5\xi^2-6\xi+2)(\xi^2-\xi+1)=0$.
Since $\eta\neq\frac{-\xi^2}{4(2\xi-1)^2}$, $\xi^2-\xi+1=0$.
Since 
\begin{eqnarray*}
a_0(z_1z_2)&\in&\mathbb{F}p_1+\mathbb{F}p_{2,0}+\mathbb{F}p_{2,1}+\sum_{i=-2}^1\mathbb{F}a_i+\mathbb{F}a_3\\
&&-\frac{\eta^2(2\eta-1)(\xi-2\eta)(\xi^2+8\xi\eta-4\eta)}{4(\xi-4\eta)(\xi^2-6\xi\eta+4\eta-\xi)}a_2,
\end{eqnarray*}
$p_{2,0}\in\mathbb{F}p_1+\sum_{i\in\mathbb{Z}}\mathbb{F}a_i$.
Let $\xi=\frac{1+\sqrt{-3}}{2}$.
Then $a_3+\frac{1+\sqrt{-3}}{2}(a_2+a_{-2})+\frac{1-\sqrt{-3}}{2}(a_1+a_{-1})+\rho_2a_0\in \eisp{a_0}{0}$ for some $\rho_2\in\mathbb{F}$.
If $\dim M=7$, then $a_2+a_{-2}+\nu_1(a_1+a_{-1})+\rho_1a_0\in \eisp{a_0}{\xi}$ or $a_3+\mu_1(a_2+a_{-2})+\nu_1(a_1+a_{-1})+\rho_1a_0\in \eisp{a_0}{\xi}$ for some $\mu_1,\nu_1,\rho_1\in\mathbb{F}$.
Then we can compute $p_{2,0}$ and $p_{3,0}$ by using these parameters.
But no parameter satisfies the appropriate invariance of $p_{2,0}$ and $p_{3,0}$.
Hence $\dim M=6$.
However, by a similar argument, the dimension of $M$ cannot be 6.
Hence $M$ cannot be an axial algebra of Majorana type in this case.
\end{proof}
Since $\xi\neq\eta,2\eta$ and $\eta\neq1$, $\xi\notin\{\frac{4}{9},\frac{2}{5},-4\pm2\sqrt{5}\}$.
Thus if $M$ is isomorphic to $\rmsix_2(\frac{2}{3},\frac{-1}{3})^{\times}$, then $\chf{F}\neq3$.
If $\mathbb{F}\neq3,11$, then both of $\rmsix_2(\frac{1\pm\sqrt{97}}{24},\frac{53\pm5\sqrt{97}}{192})^{\times}$ are appropriate.
In the case when $\mathbb{F}=11$, $(\frac{1\pm\sqrt{97}}{24},\frac{53\pm5\sqrt{97}}{192})=(2,7)$ or $(-1,1)$.
Thus only $\rmsix_2(2,7)^{\times}$ is appropriate.
Hence if $M$ is a quotient of $\sixtwo{\xi}$, $M$ is isomorphic to $\rmsix_2(\frac{2}{3},\frac{-1}{3})^{\times}$ with $\chf{F}\neq3$, $\rmsix_2(\frac{1\pm\sqrt{97}}{24},\frac{53\pm5\sqrt{97}}{192})^{\times}$ with $\chf{F}\neq3,11$ or $\rmsix_2(2,7)^{\times}$ with $\chf{F}=11$.

Thus Step \ref{prop6} is proved and then the proof of Main Theorem is completed.

\section*{Acknowledgments}
The author is grateful to Prof.\ Atsushi Matsuo for guidance throughout the work.

\end{document}